\documentclass[11pt]{amsart}

\usepackage[utf8]{inputenc}
\usepackage[french,spanish,english]{babel}

\usepackage{amsfonts, amsthm, amssymb, amsmath}
\usepackage{pgf,tikz}
\usetikzlibrary{arrows,shapes,calc,positioning,matrix,external}
\usepackage{tikz-cd}

\usepackage[cmtip,all]{xy}
\xyoption{arc}
\usepackage{graphicx}
\usepackage{subfigure}
\usepackage{xcolor}
\usepackage{caption}

\usepackage{url}
\usepackage{hyperref}
\usepackage{microtype}

\usepackage{bm,dsfont,epstopdf,enumitem,mathrsfs,mathabx, multirow,float}
\usepackage{amsopn, amsgen, amscd}
\usepackage{verbatim}
\usepackage{fancybox}

\newcommand{\R}{\mathbb{R}}
\newcommand{\N}{\mathbb{N}}
\newcommand{\Z}{\mathbb{Z}}
\newcommand{\Q}{\mathbb{Q}}
\newcommand{\C}{\mathbb{C}}
\renewcommand{\H}{\mathbb{H}}
\newcommand{\D}{\mathbb{D}}
\newcommand{\bS}{\mathbb{S}}

\newcommand{\Sol}{\mathcal{H}^2}
\newcommand{\M}{\mathcal{M}}
\newcommand{\SolU}{\mathcal{S}^1}
\newcommand{\X}{\mathcal{X}}
\newcommand{\Y}{\mathcal{Y}}
\newcommand{\K}{\mathcal{K}}
\newcommand{\T}{\mathcal{T}}

\renewcommand{\circ}{\raisebox{-1ex}{$\mathaccent23{\phantom{x}}$}}
\newcommand{\bs}{\backslash}

\newcommand{\PSL}{{\rm PSL}(2,\R)}
\newcommand{\PSLZ}{{\rm PSL}(2,\Z)}

\newcommand{\SLZ}{{\rm SL}(2,\Z)}

\renewcommand{\mod}{~\mathrm{mod}~}

\newcommand{\matriz}[4]{\left(
\begin{array}{cc}
 #1 & #2 \\
 #3 & #4
\end{array}
\right)}

\newtheorem{theorem}{Theorem}
\newtheorem{lemma}{Lemma}

\newtheorem{proposition}{Proposition}
\newtheorem*{conjecture*}{Conjecture}

\theoremstyle{definition}
\newtheorem{definition}{Definition}
\newtheorem{example}{Example}

\newtheorem{remark}{Remark}

\newtheoremstyle{named}{}{}{\itshape}{}{\bfseries}{.}{.5em}{\thmnote{#3} #1}
\theoremstyle{named}
\newtheorem{namedtheorem}{Theorem}

\title[Horocyclic trajectories in hyperbolic solenoidal surfaces]{Horocyclic trajectories in hyperbolic solenoidal surfaces of finite type}  
\author[F. Alcalde]{Fernando Alcalde Cuesta} 
\address{Universidad de Santiago de Compostela, 15782 Santiago de Compostela, Spain.}
\email{fernando.alcaldecuesta@gmail.com}
\author[A. Carballido]{\'Alvaro Carballido Costas}
\address{Galicia Supercomputing Center (CESGA), Avenida de Vigo s/n, 15705 Santiago de Compostela, Spain.}
\email{alvarocarballidocostas@gmail.com}
\author[M. Mart\'{\i}nez]{Matilde Mart\'{\i}nez}
\address{Instituto de Matem\'atica y Estad\'{\i}stica Rafael Laguardia, 
         Facultad de Ingenier\'{\i}a, Universidad de la Rep\'ublica, 
         J.Herrera y Reissig 565, C.P.11300 Montevideo, Uruguay.}
\email{matildem@fing.edu.uy}
\author[A. Verjovsky]{Alberto Verjovsky}
\address{Instituto de Matem\'aticas, Universidad Nacional Aut\'onoma de M\'exico, Apartado
Postal 273, Admon. de correos \#3, C.P. 62251 Cuernavaca, Morelos, M\'exico.}
\email{alberto@matcuer.unam.mx}

\thanks{The visits of M.M. and A.V. to research centres in Santiago de Compostela and Valladolid have been partially funded by State Research Agency PID2019-105621GB-100 and Xunta de Galicia 2023 GRC GI-2136 from Spain. M.M. and A.V. would also like to acknowledge Grupo CSIC 883174 (UdelaR, Uruguay) and Proyecto PAPIIT IN103324 (DGAPA, UNAM, México) for their financial support.}

\begin{document}

\begin{abstract}
We study the dynamical properties of the laminated horocycle flow on the unit tangent bundles of 2-dimensional smooth solenoidal manifolds of finite type. These laminations are the analog of complete hyperbolic surfaces of finite area.
\end{abstract}

\maketitle

\section{Introduction}
 
Geodesic and horocyclic trajectories on surfaces have been studied in depth since the end of the 19th century, especially in relation to the problem of the existence of periodic trajectories, with foundational contributions from J. Hadamard and H. Poincaré. The existence of non-periodic recurrent geodesics and transitive geodesics on surfaces of negative constant curvature dates back to the work of H. M. Morse and G. D. Birkhoff in the 1920s. It is the germ of the development of the ergodic theory in the 1930s. 

In 1936, in a seminal paper about the horocycle flow \cite{Hedlund}, G. A. Hedlund proved that all horocyclic trajectories in the unit tangent bundle of a hyper\-bolic surface of finite area are either periodic or dense. In particular, all horocyclic trajectories are dense if the surface is compact. Almost four decades passed before H. Furstenberg \cite{Furstenberg2} proved a stronger result in the compact case, namely unique ergodicity of the horocycle flow. 

For a hyperbolic surface of finite area, every periodic trajectory supports an invariant probability measure, and S. G. Dani \cite{Dani1978} for the modular surface and later S. G. Dani and J. Smillie \cite{DaniSmillie} for any hyperbolic surface of finite area proved that the Liouville measure on the unit tangent bundle and these measures are the unique measures which are invariant and ergodic for the horocycle flow (up to scaling). 

The distribution of the closed trajectories with large periods was studied by P. Sarnak in \cite{Sarnak1981} showing that the (normalised) Liouville measure is the weak* limit of the measures $(g_t)_*\nu$ which are obtained from a probability measure $\nu$ supported by a closed horocyclic trajectory pushing it forward by the geodesic flow $g_t$ for negative times. 

Over the last forty years, the topological and measurable dynamics of the horocycle flow have become well understood for geometrically finite surfaces \cite{Burger, Eberlein, 
Ratner1990, Ratner1992, Ratner1994, 
Roblin} and some classes of geometrically infinite surfaces \cite{Babillot2004,BabillotLedrappier,LedrappierSarig, MatsumotoMinimal,Sarig}. See also \cite{Dal'Bobook,Kaimanovich2000,Starkov1995} for a general overwiew. While horocyclic orbits are still periodic or dense in restriction to the non-wandering set when the surface is geometrically finite \cite{Eberlein}, the situation is very different in the geometrically infinite case: there are geometrically infinite surfaces with horocyclic orbits that are neither periodic nor dense  \cite{Dal'BoStarkov2000,Pommerenke1981} and other with horocyclic orbits that are neither closed nor dense \cite{Dal'BoStarkov2000,Nicholls,Sullivan1981} in the non-wandering set. In fact, there are many unresolved problems in connection with the nature of orbit closures and the existence of minimal sets. There have been some recent advances \cite{Dal'Bo2025} that should be added to several partial answers \cite{Bellis2018,Farre2024,GayeLo,Kulikov,MatsumotoMinimal}. 

In this paper we will study the topological dynamics of the horocycle flow on hyperbolic solenoidal surfaces of finite type, that is, transversely modelled by the Cantor set and based on surfaces of negative constant curvature with a finite number of cusps. 

The paradigm of this type of hyperbolic solenoidal surfaces is the 
solenoid $\X$ obtained as the inverse limit of all finite regular coverings of the Riemann sphere minus three points $S=\C P^1-\{0,1,\infty\}$. Since $S$ is the branched cover
of order 6 of the modular orbifold $M =\PSLZ \bs \H$
(corresponding to the principal congruence subgroup $\Gamma(2)$), there is a holomorphic map (in the sense of holomorphic laminations) $\pi: \X \to M$, cf. Example~\ref{ex:arithmetic}. Therefore, this marvellous solenoidal hyperbolic surface
codifies many aspects of mathematics, for instance: the absolute Galois group $\text{Gal}(\overline{\Q}/\Q)$ of $\Q$, the theory of Belyi arithmetic surfaces
defined over $\overline{\Q}$ and therefore Grothendieck's project 
\emph{Esquisse d'un programme} \cite{Grothendieck1997}, as was observed by P. Deligne in his celebrated paper 
\emph{Le groupe fondamental de la droite projective moins trois points}\cite{Deligne}.
The geodesic flow $g_t$ on the unit tangent bundle $\Y$ of $\X$ acts on each minimal set of the horocycle flow. This family is projected onto a one-parameter family on closed horocyclic trajectories over horocycles on the modular surface whose periods become increasingly large or small depending on whether the time $t$ goes to $-\infty$ or $+\infty$. The rate of approach of their lengths to the area form is related to the Riemann hypothesis, as observed by D. Zagier \cite{Zagier1981} (see also \cite{Sarnak1981}).

This motivating example is the heart of this work, although we will focus on the topological dynamics of the horocycle flow on hyperbolic solenoidal surfaces and leave the measurable aspects for future work. What are the orbits of the horocycle flow on a hyperbolic solenoidal surface of finite type? What is the solenoidal structure over each cusp? 
Where do the minimal compact sets converge when we make the periods of their projections larger or smaller by applying the geodesic flow?
These are the three questions we intend to answer. 

Solenoidal surfaces of finite type are a broad class, where the horocycle flow exhibits too wide a range of dynamical behaviours. To illustrate this, using the fact that the group of homeomorphisms of the Cantor set is perfect \cite{Anderson}, we will give a family of examples of  hyperbolic solenoidal surfaces with a single cusp proving that any dynamics on the Cantor set is realizable over the cusp (see Example~\ref{example:dynamics_over_cusp_can_be_anything}). In particular, there exist horocyclic orbits which are neither periodic nor dense, and in fact any 1-dimensional solenoid can appear as a horocycle orbit closure. Thus some additional property is needed to have a clear-cut description of closed sets invariant under the horocycle flow.

Here we will follow M. C. McCord \cite{McCord} in emphasising solenoidal surfaces obtained as an inverse limit of surface coverings. If all of these coverings are regular, these inverse limits are \emph{topologically homogeneous}, that is, for any two points, there is a global homeomorphism that sends one point to the other one. According to \cite{ClarkHurder}, this property characterises (up to homeomorphism) the inverse limits of regular coverings, named \emph{McCord solenoids}, in the compact case. Even if this homogeneous case is interesting, most of our results apply to the more general setting of inverse limits of not necessarily regular surface coverings. Therefore, they cannot be obtained through variations of Ratner's theorems. As proved in \cite{Alvarez-Brum-Martinez-Potrie}, there are transversely Cantor compact laminations by hyperbolic surfaces with all possibles topologies showing a behaviour which is far away from the homogeneous context. Based on the celebrated results of A. Eskin, M. Mirzakhani, and A. Mohammadi, the dynamics of horocycle flows in moduli spaces of translation surfaces have often been presented as parallel to that of unipotent flows in homogeneous spaces. However, the \lq emerging picture' of the horocycle flow in moduli spaces drawn by J. Chaika and B. Weis in \cite{ChaikaWeiss} is more complicated than this parallelism might suggest. In fact, translation surfaces are a valuable source of examples of solenoidal manifolds with remarkable properties, which we will study in a future contribution.

The first two sections will be devoted to introducing the notion of \emph{n-dimensional solenoidal manifold} \cite{Sullivan1993,Sullivan2014, VerjovskySolenoids,VerjovskySullivan} and describing examples to illustrate properties that characterise some remarkable families. All solenoidal manifolds considered here will be smooth, that is, of class $C^\infty$ (although $C^3$ is actually sufficient for our purposes).

The aim of Section~\ref{SHedlund} will be precisely to prove a Hedlund theorem for any hyperbolic solenoidal surface of finite type $\X$ that is the inverse limit of a tower of coverings of surfaces of finite area. (This includes our motivating example.)

\begin{theorem}\label{thm:1}
Let $\X$ be the inverse limit of a tower of finite coverings based on a hyperbolic surface $S$ of finite area. Then any orbit of the horocycle flow on the unit tangent bundle $\Y = T^1 \X$ of $\X$ is either dense or projects onto a periodic orbit of the horocycle flow on the unit tangent bundle $T^1S$ of $S$.
\end{theorem}

In Section~\ref{Sminimalsets}, we will analyse the nature of closed subsets $\M$ of $\Y = T^1 \X$ which are invariant by the horocycle flow and  project onto closed horocyclic orbits in $T^1 S$ associated to the cups of $S$. Our example above is precisely of type (3):

\begin{theorem}\label{thm:2}
Let $\X$ be the inverse limit of a tower of finite coverings of a hyperbolic surface $S$ of finite area. Let $\M$ be a compact subset of $\Y = T^1 \X$ which is invariant by the horocycle flow. Then there are three possibilities: 
\smallskip 

\noindent
(1) $\M$ is a 1-dimensional solenoidal manifold that is a minimal set of the horocycle flow on $\Y$,
\smallskip 

\noindent
(2) $\M$ is a disjoint union of finitely many 1-dimensional solenoidal manifolds that are minimal sets of the horocycle flow on $\Y$,
\smallskip  

\noindent
(3) $\M$ is a disjoint union of uncountably many 1-dimensional solenoidal manifolds (possibly reduced to closed horocyclic orbits) that are minimal sets of the horocycle flow on $\Y$.
\end{theorem}

\noindent
The remainder of Section~\ref{Sminimalsets} will be devoted to providing examples of McCord solenoidal surfaces that satisfy each of these conditions. Note that the regu\-larity of the covering maps involved in the inverse limits is not necessary to prove this result, which will remain valid for any inverse limit of coverings.

In Section~\ref{Slargeperiods}, we will study the behaviour of minimal sets $\M$ of the horocycle flow in $\Y = T^1 \X$ that project over closed orbits in $T^1 S$ associated to horocycles with small and large period. Under the hypotheses of previous theorems, from the tower of coverings defining the inverse limit, we directly construct a compactification $\widehat{\Y}$ of $\Y$. It will be homeomorphic to the end compactification of $\Y$, and we will call \emph{cuspidal ends}  the boundary points in $\widehat{\Y} - \Y$. 

Our last theorem is proved under the additional assumption that the solenoidal manifold $\Y$ --equivalently, the solenoidal surface $\X$-- is McCord. 
In this setting, we will see that if we take horocycles with smaller and smaller period for positive times, the minimal sets $g_t(\M)$ converge (in the Hausdorff distance) to a cuspidal end of $\Y$. Taking horocycles of increasingly larger period, we will obtain the following topological version of Sarnak's theorem for solenoidal surfaces: 

\begin{theorem}\label{thm:3}
Let $\X$ be a McCord hyperbolic solenoidal surface of finite type. Let $\widehat{\Y}$ be the end compactification of its unit tangent bundle $\Y=T^1\X$. If $\M$ is a minimal set of the horocycle flow on $\Y$, then
\smallskip 

\noindent
(i) for positive times, the minimal sets $g_t(\M)$ converge (in the Hausdorff distance) to a cuspidal end of $\Y$,
\smallskip 

\noindent
(ii) for negative times, the minimal sets $g_t(\M)$ converge (in the Hausdorff distance) to the whole of $\widehat{\Y}$.
\end{theorem}  

The authors are grateful for the insightful comments and suggestions of the anonymous referee.

\setcounter{theorem}{0}

\section{Solenoidal surfaces}

\subsection{Definition of solenoidal manifolds} 

\begin{definition}\label{def:solenoidal_manifold}
An $n$-dimensional \emph{solenoidal manifold} is a second countable locally compact Hausdorff space that is locally homeomorphic to a disk of dimension $n$ times a Cantor set. 
Namely, $\X$ is a solenoidal manifold if it has an open cover $\{U_i\}_{i\in I}$ such that:
\begin{itemize}
    \item [\textbullet] for every $i$, there is a homeomorphism $\varphi_i:  U_i \to \D^n\times T_i$, where $\D^n$ is the unit open disk in $\R^n$ and $T_i$ is compact, totally disconnected and perfect;
    \item [\textbullet] for $(x,t)\in \varphi_j(U_i\cap U_j)$, $\varphi_i \circ \varphi_j^{-1}(x,t)=(a_{ij}(x,t),b_{ij}(t))$;
    \item [\textbullet] for fixed $t$, the functions $a_{ij}(\cdot,t)$ are smooth.
Smooth means $C^\infty$, but class $C^3$ is actually sufficient to all effects in this paper.
\end{itemize}
\end{definition}

\noindent
The $(U_i, \varphi_i)$ are called \emph{foliated charts}.
The second condition implies that the connected components of $\X$, which we call \emph{leaves}, have a structure of $n$-dimensional manifolds. We say that the sets of the form $\varphi_i^{-1}(\D^n \times \{t\})$ are \emph{plaques}, and
$T'_i=\varphi_i^{-1}(\{0\}\times T_i)$ are \emph{local transversals}. The union $T=\bigcup_{i\in I}T'_i$ of local transversals is a \emph{complete transversal} that meets all the leaves, and each $T'_i$ is a clopen set of $T$. In fact, we can refine the foliated atlas of $\X$ in order to assume the equivalence relation defined by the plaques on $T$ is trivial, so $T$ is the Cantor set decomposed into clopen sets $T_i$.

\begin{definition} \label{def:minimal_solenoidal_manifold}
A solenoidal manifold $\X$ is \emph{minimal} if every leaf is dense in $\X$.
\end{definition}

\subsection{The universal solenoid}

A well-known example of a minimal solenoidal surface is the \emph{universal solenoid}, which we call $\SolU$. We 
give two alternative descriptions of $\SolU$, reflecting two points of view relevant to the sequel. 
\medskip 

Consider the circle $\bS^1=\{z\in \C\ :\ |z|=1\}$ and the family of its finite coverings
$$\pi_n : z \in \bS^1 \mapsto z^n \in \bS^1$$
 for 
$n \geq 2$.
It is a directed set with the following order relation: 
$$\pi_n\leq \pi_m$$
if and only if there exists a covering $\pi:\bS^1\to \bS^1$ such that $\pi_m=\pi_n \circ\pi$. Namely, 
$$
\pi_n\leq \pi_m\ \Longleftrightarrow \ n \mbox{ divides } m,$$
and in this case  $\pi=\pi_{\frac{m}{n}}$.

The \emph{inverse limit} of this family is 
$$\SolU=\left\{z=(z_1,z_2,\ldots)\in (\bS^1)^{\N}\ :\ \left(n \mbox{ divides } m \Rightarrow \pi_{\frac{m}{n}}(z_m)=z_n \right)\right\}.$$
It is a closed subset of $(\bS^1)^{\N}$, therefore compact. 
There is a projection $\pi:\SolU\to \bS^1$ given by $\pi(z)=z_1$. The preimage $\pi^{-1}(1)$ is a Cantor set, and it is a local (complete) transversal of $\SolU$ in the sense of Definition \ref{def:solenoidal_manifold}. Locally, $\SolU$ is a Cantor set times a segment. All leaves are dense, so $\SolU$ is a one-dimensional minimal solenoidal manifold. 

As explained in \cite{BurgosVerjovsky}, the universal solenoid $\SolU$ also is the Pontryagin dual of the additive group $\Q$ with the discrete topology, that is, the group of homomorphisms $\hat \Q = Hom(\Q,\bS^1)$ with the compact-open topology. 
\medskip 

It has an alternative description. Let $\hat{\Z}$ be the profinite completion of $\Z$. It is a Cantor group and $\Z$ is naturally embedded into $\hat \Z$ as a dense subgroup. Then 
$$\SolU= \Z \bs (\R\times \hat\Z),$$
where the action on $\Z$ on $\R\times \Z$ is diagonal: by left translations on $\R$ and  
$\hat\Z$.

Recall that $\hat\Z$ is the set of sequences
$$
(k_1 \mod 1, k_2 \mod 2, k_3 \mod 3,\ldots)
$$
such that whenever $n$ is a multiple of $m$, $k_m=k_n\mod m$.
The embedding of $\Z$ is simply given by
$$
k\in\Z \mapsto (k \mod 1, k \mod  2, k \mod  3,\ldots).
$$
Let $\tau:\hat\Z \to \hat\Z$ be the translation by 1, namely, 
$$
\tau:(k_1 \mod 1, k_2 \mod  2,\ldots)\mapsto (k_1+1 \mod 1, k_2+1 \mod  2,\ldots).
$$
Then the above paragraph says that $\SolU$ is the suspension of $\tau$.

\begin{remark}
    When defining $\SolU$ as an inverse limit, we could use any cofinal set of the directed set $\{\pi_n:\bS^1\to\bS^1\ :\ n\geq 2\}$, for example, 
$$\{\pi_{n!}:\bS^1\to\bS^1\ :\ n\geq 2\}.$$
This second choice has the advantage of being totally ordered. We say it is a \emph{tower of coverings}, and we write
    $$\cdots \bS^1 \overset{\pi_n}{\longrightarrow} \bS^1 \longrightarrow \cdots \longrightarrow \bS^1 \overset{\pi_2}{\longrightarrow} \bS^1.$$ 
For simplicity and conciseness, we will write
$$
\SolU= \varprojlim \, \big(\bS^1\overset{\pi_n}{\longrightarrow}\bS^1\big).
$$
Notice that the map $\pi_{n!}$ is obtained as the composition of the first $n-1$ arrows. It means that the directed set we are considering in order to take its inverse limit is the totally ordered set $\{\pi_n\circ \pi_{n-1}\circ\cdots\circ\pi_2:\bS^1\to\bS^1\ :\ n\geq 2\}$.
\end{remark}

\subsection{The universal hyperbolic solenoid}

A $2$-dimensional solenoidal manifold will be called a \emph{solenoidal surface}.
A well-known example of a minimal solenoidal surface is the \emph{universal hyperbolic solenoid}, which we call $\Sol$. Just as we did with $\SolU$, we will give two alternative descriptions.
\medskip 

Let $S_0$ be any compact hyperbolic surface, and $x_0\in S_0$. We consider the family $A$ of all 
pointed hyperbolic surfaces 
$(S_\alpha,x_\alpha)$ which are regular coverings of $S_0$ such that the covering sends $x_\alpha$ to $x_0$. Namely, $S_\alpha$ is a compact hyperbolic surface and there is a covering map $\pi_\alpha:S_\alpha\to S_0$ such that $\pi_\alpha(x_\alpha)=x_0$ and the group of deck transformations acts transitively on the preimages of any point in $S_0$.

A partial order can be defined in $A$ by stating that $\alpha\leq \beta$ if there exists a finite regular covering
$$\pi_{\alpha\beta}:S_\beta \to S_\alpha$$
such that $\pi_{\alpha\beta}(x_\beta)=x_\alpha$.

The inverse limit of $(A,\leq)$ is 
$$\Sol=\left\{y=(y_\alpha)\in \prod_{\alpha\in A} S_\alpha\ :\ \pi_{\alpha\beta} (y_\beta)=y_\alpha \mbox{ whenever } \alpha\leq \beta\right\},$$
seen as a topological subspace of $\prod_{\alpha\in A} S_\alpha$. It is a minimal compact solenoidal surface with simply connected leaves \cite{Sullivan1993}, which does not depend on the choice of $S_0$. 

There is a projection $\pi:\Sol\to S_0$ defined by $\pi(y)=y_0$. The preimage $\pi^{-1}(x_0)$ is a Cantor set that can be identified as the set of all the 
 base points $x_\alpha$, $\alpha\in A$. It is a 
local (complete) transversal in the sense of 
Definition \ref{def:solenoidal_manifold}, and $\Sol$ is a minimal solenoidal surface.
\medskip

An alternative description of $\Sol$ is the following. As before, $S_0$ is a compact hyperbolic surface, and $x_0\in S_0$. Let $G_0=\pi_1(S_0,x_0)$ be the fundamental group of $S_0$, and $G$ the profinite completion of $G_0$. Therefore, $G$ is a topological Cantor group, and $G_0$ injects naturally into $G$ as a dense subgroup. Then, if $\H$ is the hyperbolic plane,
$$\Sol=G_0 \bs (\H\times G),$$
where the group $G_0$ acts on the product $\H\times G$ diagonally, by deck transformations on $\H$ and left translation on $G$. The action is properly discontinuous, and the quotient is the universal hyperbolic solenoid. The leaves are the images of the horizontal foliation in $\H\times G$. Each vertical copy of $G$, namely, the set $\{z\}\times G$ with $z\in\H$, injects into the quotient, becoming a (complete) local transversal. Since the action of $G_0$ preserves the group structure on $G$, these transversals are not only Cantor sets, but they carry a group structure preserved under the change of foliated charts. 
\medskip

In the universal hyperbolic solenoid $\Sol$, leaves carry a hyperbolic structure. 
If we see $\Sol$ as the inverse limit of coverings, the restriction of the projection $\pi$ to any leaf $L$ gives a universal covering of $S_0$. The hyperbolic metric on $S_0$ can be pulled back to $L$ via $\pi$. This endows each leaf with a hyperbolic metric, and these metrics vary continuously, in the smooth topology, as we change leaves in $\Sol$.

If we see $\Sol$ as a quotient of $\H\times G$, then $G_0$ can be thought of as a Fuchsian group, so that $S_0=G_0\bs\H$ and the action of $G_0$ on $\H$ by deck transformations is an action by isometries of the hyperbolic metric. Therefore, the metric on horizontal leaves $\H\times \{g\}$ in $\H\times G$ induces a metric on leaves in $\Sol$.
\medskip

When $\Sol$ is thus endowed with a metric, all kinds of objects from Riemannian geometry can be defined and studied. This is true more generally, so we will give some general definitions.

\subsection{Hyperbolic solenoidal surfaces}

If $\X$ is a $n$-dimensional solenoidal manifold, the \emph{tangent space} $T\X$ of $\X$ is a $2n$-dimensional solenoidal manifold whose leaves are the tangent bundles of the leaves of $\X$. It is defined by extending the foliated charts of $\X$, in the same way as we define the tangent bundle of a manifold. 

Similarly, all objects in differential geometry can be defined, using the differentiable structure of the leaves. In the transverse Cantor direction, these objects are \emph{continuous}. Of special importance to us will be hyperbolic metrics. 

\begin{definition}\label{def:metric}
If $\X$ is a solenoidal manifold, a \emph{Riemannian metric} on $\X$ is a positive definite metric tensor on $T\X$ that varies smoothly along the leaves of $\X$ and continuously in $\X$ in the smooth topology.
The metric is hyperbolic if all sectional curvatures are equal to -1.
\end{definition}

All solenoidal manifolds admit Riemannian metrics, since they can be defined in local charts and glued using partitions of unity.

Under a mild topological condition, compact solenoidal surfaces admit infinitely many hyperbolic metrics, one in each conformal class. A solenoidal surface endowed with such a metric will be called a \emph{hyperbolic solenoidal surface}.

\begin{definition}\label{def:geodesic_and_horocycle_flows}
Let $\X$ be a hyperbolic solenoidal surface and $\Y=T^1\X$ be its unit tangent bundle. The \emph{geodesic flow} $g_t$ is the flow on $\Y$ for which all leaves are invariant, and that restricts to the geodesic flow on each leaf. The (stable) horocycle flow $h_t$ is the flow on $\Y$ for which all leaves are invariant, and that restricts to the stable horocycle flow on each leaf.
\end{definition}

For example, the \emph{unit tangent bundle} $T^1\Sol$ is the three-dimensional solenoidal manifold
$$T^1\Sol = G_0\bs(T^1\H\times G),$$
whose leaves are the unit tangent bundles of the leaves of $\Sol$. The foliated geodesic and horocycle flows can be defined on $T^1\Sol$. When restricted to each leaf, they are the usual geodesic and horocycle flows, and they project via $\pi$ to the corresponding flows in $S_0$.

Each leaf of $\Sol$ is simply connected, namely, it is a hyperbolic plane. Therefore, the geodesic and horocycle flows, when restricted to a single leaf, are well understood. Their complexity comes from the fact that leaves are not embedded in $\Sol$.
As every leaf accumulates everywhere, this forces a lot of recurrence on the flows $g_t$ and $h_t$.

In fact, as in the case of a compact hyperbolic surface, the horocycle flow in $\Sol$ is minimal; see \cite{Martinez-Matsumoto-Verjovsky}.

\subsection{Hyperbolic solenoidal surfaces of finite type}

We will concentrate on solenoidal surfaces that are, in some way, similar to non-compact hyperbolic surfaces of finite area.

\begin{definition}\label{def:solenoidal_surface_of_finite_type}
Let $\X$ be a hyperbolic solenoidal surface. We say that it is \emph{of finite type} if there exist a non-compact hyperbolic surface $S_0$ of finite area and a map $\pi:\X\to S_0$ such that
\begin{itemize}
    \item [\textbullet] $\pi$ is a locally trivial fibration whose fibre is a Cantor set and
    \item [\textbullet] each fibre $\pi^{-1}(x)$ is a local transversal of $\X$.
    \item the restriction of $\pi$ to any leaf of $\X$ is a local isometry.
\end{itemize}
\end{definition}

Remark that, since we demand that $S_0$ be non-compact, the universal solenoid $\Sol$ is not included in this definition. In the next sections we will describe several examples of hyperbolic solenoidal surfaces of finite type, with different kinds of transverse structures.

\section{Examples of hyperbolic solenoidal surfaces of finite type}

\subsection{Solenoidal cusps}

We can easily endow the punctured open disk 
$$\D^*=\{z\in\C\ :\ 0<|z|<1\}$$
with a hyperbolic metric in such a way that it becomes a hyperbolic cylinder with a \emph{cusp}. In fact, if $\H$ be the upper half-plane model of the hyperbolic plane, then the transformation group of the holomorphic covering $h:\H\to\D^*$ given by $h(z)=e^{iz}$ preserves the hyperbolic metric on $\H$.
With the induced metric, $\D^*$ is a complete hyperbolic surface for which the radial rays are geodesics and the circles $|z|=k$ (for $k>0$) are horocycles whose length tends to 0 as $k\to 0$. All regular coverings $\pi_n:\D^\ast \to \D^\ast$ given by $\pi_n(z)=z^n$ preserve this hyperbolic metric.

\begin{definition}\label{def:solenoidal_cusp} 
The inverse limit of the regular coverings $\pi_n:\D^*\to \D^*$ is a hyperbolic solenoidal surface $\mathcal{C}$, which we call \emph{solenoidal cusp}.
\end{definition}

Identifying the punctured closed disk $\bar{\D}^*$ with $\{ z\in\C\ :\ 0<|z|\leq\frac{1}{e} \} \subset \D^*$, the inverse limit $\bar{\mathcal{C}}$ of the regular coverings $\pi_n:\bar{\D}^*\to\bar{\D}^*$ can be thought of as a \emph{hyperbolic solenoidal surface with boundary}, the boundary being $\SolU$.

\begin{remark}
The transverse structure to $\SolU$ and $\mathcal{C}$ is the same, represented by the Cantor group $\hat\Z$. Both solenoidal manifolds have a completely invariant measure $\mu$, coming from the Haar measure $\nu$ on $\hat\Z$. Indeed, $\mu$ is obtained by integrating the transverse invariant measure $\nu$ against the volume form on the leaves. The mass of $\bar{\mathcal{C}}$ can be easily computed, and it is finite. This strengthens the idea that $\bar{\mathcal{C}}$ is really a solenoidal analogue of a cusp in a hyperbolic surface. 
\end{remark}

\subsection{Examples}

We will construct different hyperbolic solenoidal surfaces of finite type.
\medskip

\begin{example}[Suspension of a Schottky group]\label{example:Schottky}
    Let $\Gamma$ be a Fuchsian group such that $S=S_{0,3}=\Gamma\backslash\H$ is a surface of genus 0 with 3 punctures. It is a free group in two generators. Let $\rho:\Gamma\to \PSL$ be any faithful representation for which $\rho(\Gamma)$ is a Schottky group with a Cantor limit set $\Lambda$. Then $\rho(\Gamma)$ 
acts projectively on the circle, seen as the boundary at infinity of the hyperbolic plane, 
leaving $\Lambda$ invariant. 
Consider the suspension
    $$\X=\Gamma\bs(\H\times\Lambda),$$
    where $\Gamma$ acts diagonally, by isometries on the hyperbolic plane and as $\rho(\Gamma)$ on $\Lambda$.
    Then $\X$ is a solenoidal surface, and there is a fibration $\pi:\X\to S$ whose fibre is the Cantor set $\Lambda$.
     Namely, $\X$ is a hyperbolic solenoidal surface of finite type.

    If $H$ is a closed horocycle in $S$, its preimage $\pi^{-1}(H)$ is a one-dimensional solenoidal manifold whose leaves are horocycles in $\X$. Let $\gamma$ be the element of the fundamental group $\Gamma$ of $S$ represented by $H$, and $\tilde H$ the $\gamma$-invariant horocycle in the hyperbolic plane that projects onto $H$. Then $\rho(\gamma)$ has an attracting fixed point $\gamma^+$ and a repelling fixed point $\gamma^-$ in $\Lambda$. The horocycles $\tilde H\times \{\gamma^{\pm}\}$ in $\H\times\Lambda$ project onto two closed horocycles in $\X$ each of which projects one-to-one onto $H$. All other horocycles in $\pi^{-1}(H)$ accumulate on both closed horocycles, since the transverse dynamics of $\pi^{-1}(H)$ is that of $\rho(\gamma)$ acting on $\Lambda$.
\end{example}

\begin{example}[Any dynamics can appear over the cusp]\label{example:dynamics_over_cusp_can_be_anything}
 Let $K$ be the Cantor set. Its group of homeomorphisms $Homeo(K)$ is \emph{perfect} (see \cite{Anderson}), which means that its commutators generate it. Consider any homeomorphism $\varphi:K\to K$, and let $f_1,\ldots,f_g, h_1,\ldots,h_g\in Homeo(K)$ be such that
    $$\prod_{i=1}^g [f_i,h_i]=\varphi.$$
    
Let $S=S_{g,1}$ be a hyperbolic surface of genus $g$ with one puncture. Its fundamental group is free in the generators $\alpha_1,\ldots,\alpha_g,\beta_1\,\ldots,\beta_g$, and the product of commutators $\prod_{i=1}^g[\alpha_i,\beta_i]$ represents the class of closed horocycles. Let $\rho:\pi_1(S)\to Homeo(K)$ be the representation defined by
\begin{eqnarray*}
\rho(\alpha_i)&=&f_i\\
\rho(\beta_i)&=&h_i
\end{eqnarray*}
for all $i$. Its \emph{suspension} is the hyperbolic solenoidal surface
$$\mathcal{X}=\pi_1(S)\bs(\H\times K),$$
where $\pi_1(S)$ acts diagonally, by deck transformations on $\H$ and via the representation $\rho$ on $K$. Therefore, there is a fibration $\pi:\mathcal{X}\to S$ with fibre $K$.

If $H$ is a closed horocycle in $X$ of length 1, the element it represents in $\pi_1(S)$ acts on $K$ like $\varphi$, and $\pi^{-1}(H)$ is the suspension of $\varphi$. 
This means that \emph{for any} homeomorphism $\varphi$ of $K$, there is a hyperbolic solenoidal surface of finite type for which the preimage under $\pi$ of a neighbourhood of the cusp is homeomorphic to a cylinder over the suspension of $\varphi$.
 \end{example}

\begin{example}[The cuspidal solenoid $\bar{\mathcal{C}}$ as cusp of a solenoidal surface of finite type] \label{example:suspension_with_universal_solenoid_over_closed_horocycle}
    
    A special case of the above example consists in taking $K=\hat\Z$ and $\tau:\Z\to\Z$, as before,  the homeomorphism given by 
    $$\tau:(k_1 \mod 1, k_2 \mod  2,\ldots)\mapsto (k_1+1 \mod 1, k_2+1 \mod  2,\ldots).$$

The construction of Example \ref{example:dynamics_over_cusp_can_be_anything} gives a hyperbolic solenoidal surface $\X$ and a fibration $\pi:\mathcal{X}\to S$ with fibre a Cantor set, where $S$ is a non-compact hyperbolic surface of finite area.
If $H$ is a closed horocycle in $X$ of length 1, its preimage $\pi^{-1}(H)$ is the universal solenoid $\SolU$. If $C$ is the connected component of $S-H$ that is homeomorphic to a cylinder, $\pi^{-1}(C)$ is homeomorphic to the solenoidal cusp ${\mathcal C}$.
\end{example}

\begin{remark}
 The solenoidal surface $\X$ in Example \ref{example:suspension_with_universal_solenoid_over_closed_horocycle} can be seen as the surface obtained by gluing a solenoidal cusp $\bar{\mathcal{C}}$ to a hyperbolic solenoidal surface with boundary $\SolU$. It is simply defined by suspension of the free group generated by $\alpha_1,\ldots,\alpha_g,\beta_1,\ldots,\beta_g$. As stated in  \cite{Sullivan2014,VerjovskySolenoids}, this shows that the universal solenoid $\SolU$ is cobordant to zero.
\end{remark}

\begin{example}[A solenoidal surface of finite type with a transverse $p$-adic Riemannian structure]
\label{example:suspension_with_triadic_solenoid_over_closed_horocycle}
This example comes from a construction that is similar to the previous one, but the representation is defined to preserve a transverse group structure and a transverse distance. Consider the Cantor group $\Z_3$ of the triadic integers, and 
recall that $\Z_3$ is the set of sequences
$$(k_1 \mod 3, k_2 \mod  3^2, k_3 \mod  3^3,\ldots)$$
with the $k_i\in\Z$, such that $k_i=k_j\mod 3^i$ whenever $i<j$.
The embedding of $\Z$ is simply given by
$$k\in\Z \mapsto (k \mod 3, k \mod  3^2, k \mod  3^3,\ldots).$$
Let $\tau:\Z_3 \to \Z_3$ be the translation by 1, namely, 
$$\tau:(k_1 \mod 3, k_2 \mod  3^2,\ldots)\mapsto (k_1+1 \mod 3, k_2+1 \mod  3^2,\ldots).$$

The suspension of $\tau$ gives a solenoidal 1-manifold $\mathcal{S}_3$, which can also be obtained as the inverse limit
$$\varprojlim \, \big(\bS^1\overset{\pi_n}{\longrightarrow}\bS^1\big),$$
where $\pi_n(z)=z^3$ for all $n \in \N$.

Let $S=S_{1,1}$ be a once-punctured torus with a hyperbolic metric. Its fundamental group is free in two generators $\alpha$ and $\beta$. Consider a representation $\rho:\pi_1(S)\to Aut(\Z_3)$ given by $\rho(\alpha)=\tau$ and $\rho(\beta)=\psi$, where
\begin{equation*}
\psi:(k_1 \mod 3, k_2 \mod  3^2,\ldots)\mapsto (2k_1 \mod 3, 2k_2 \mod  3^2,\ldots).
\end{equation*}

Notice that both $\tau$ and $\psi$ are group automorphisms of $\Z_3$. For $\tau$ this is clear, and for $\psi$ it follows from the fact that 2 and 3 are relatively prime. 

A straightforward computation shows that $\psi\circ\tau=\tau^2\circ\psi$, and therefore
\begin{equation}\label{equation:commutation_of_psi_and_tau}
[\psi,\tau]=\tau.
\end{equation}

The suspension of $\rho$ is the hyperbolic solenoidal surface
$$\mathcal{X}=\pi_1(S)\bs(\H\times \Z_3),$$
where $\pi_1(S)$ acts diagonally, by deck transformations on $\H$ and via the representation $\rho$ on $\Z_3$. Therefore, there is a fibration $\pi:\mathcal{X} \to S$ with fibre $\Z_3$. Since $\rho(\pi_1(S))$ acts by group automorphisms of $\Z_3$, the transversal to $\mathcal{X}$ is a Cantor set having the group structure of $\Z_3$. Furthermore, Equation \eqref{equation:commutation_of_psi_and_tau} says in particular that the group $\rho(\pi_1(S))$ is solvable, and therefore amenable. This implies that it has an invariant measure on $\Z_3$. It contains a dense subgroup of the group of translations, so this invariant measure has to be the Haar measure. It also admits an invariant distance. 

As in previous examples, we can take a closed horocycle $H$ in $S$ and look at the preimage $\pi^{-1}(H)$. It is the suspension of $\rho([\alpha,\beta])=\tau$, namely, the triadic solenoid $\mathcal{S}_3$. Moreover, if $C$ is the connected component of $S-H$ that is homeomorphic to a cylinder, $\pi^{-1}(C)$ is homeomorphic to the inverse limit
$$\mathcal{C}_3=\varprojlim \, \big( \D^*\overset{\pi_n}{\longrightarrow}\D^*\big).$$

\end{example}

\begin{example}
 We can do a similar construction over a non-compact surface $S=S_{1,n}$ of genus one and $n$ cusps, for any $n$. 
 
 Let $c_1,\ldots,c_n$ be the punctures in $S$ and $\gamma_i$ be a curve that turns once around $c_i$ (and is freely homotopic to $c_i$) for every $i$. The fundamental group of $S$ has a canonical presentation as a free group in the generators $\alpha, \beta, \gamma_1,\ldots,\gamma_{n-1}$, where $\alpha$ and $\beta$ satisfy
$$[\alpha,\beta]=\gamma_1\gamma_1\cdots\gamma_n.$$

Let $p$ be a prime number that does not divide $n+1$. For the group of $p$-adic integers $\Z_p$, we define the automorphisms $\tau$ and $\psi$ by 

\begin{align*}
 \tau(k_1\mod p,k_2\mod\,p^2,\ldots ) & = (k_1+1\mod p,k_2+1\mod\,p^2,\ldots ) \\
 \psi(k_1\mod p,k_2\mod\,p^2,\ldots ) & = ((n+1)k_1\mod p,(n+1)k_2\mod\,p^2,\ldots ).
\end{align*}
A straightforward computation shows that 
$$[\psi,\tau]=\tau^n.$$
This allows us to define a representation $\rho:\pi_1(S)\to \Z_p$ by 
\begin{align*}
 \rho(\alpha) &=\tau\\
 \rho(\beta) & = \psi\\
 \rho(\gamma_i) &= \tau \mbox{ for } i=1,\ldots,n-1.
\end{align*}
It is easy to see that $\rho(\gamma_n)=\tau$.

The suspension of this representation gives a solenoidal manifold $\X$ with a fibration $\pi:\X\to S$. If $H$ is a closed horocycle in $S$, $\pi^{-1}(H)$ is homeomorphic to the suspension $\mathcal{S}_p$ of $\tau$, called the \emph{$p$-adic solenoid}. Namely, the preimage under $\pi$ of a cusp is homeomorphic to a cylinder over $\mathcal{S}_p$. Moreover, $\pi$ is a principal bundle, since $\rho(\pi_1(S))$ acts on $\Z_p$ by group automorphisms. 
\end{example}

\subsection{Inverse limits and McCord solenoids}

If $M$ is an $n$-dimensional manifold, a  projective system of finite coverings of $M$ is an infinite family of $n$-dimensional manifolds $M_\alpha$ and finite regular coverings $\pi_\alpha:M_\alpha\to M$ such that the $\alpha$ belong to a set $A$ equipped with the order $\alpha\leq \beta$ if and only if there exists a finite covering 
$$\pi_{\alpha\beta}:M_\beta\to M_\alpha$$
such that $\pi_\beta= \pi_\alpha \circ \pi_{\alpha\beta}$.
The inverse limit 
of $\{(M_\alpha, \pi_\alpha)\}_{\alpha\in A}$, given by
$$\X=\left\{x=(x_\alpha)\in \prod_{\alpha\in A}\ :\ \pi_{\alpha\beta}(x_\beta)=x_\alpha \mbox{ if }\alpha\leq \beta\right\},$$
is an $n$-dimensional solenoidal manifold. There is a natural projection 
$\pi:\X\to M$ defined by $\pi(x)=\pi_{\alpha}(x_\alpha)$ for any $\alpha\in A$.

Since $A$ is a countable set, it has a cofinite totally ordered subset \linebreak $A_0=\{\alpha_n\ :\ n\in \N\}$ with $\alpha_n<\alpha_{n+1}$ for all $n$. Therefore, $\X$ can always be seen as the inverse limit of a \emph{tower} of finite coverings
$$\cdots \to M_{\alpha_n}\to\cdots \to M_{\alpha_1}\to M_{\alpha_0}.$$

\begin{definition}
A solenoidal manifold obtained as the inverse limit of a projective limit of finite \emph{regular} coverings of a connected manifold $M$ is a \emph{McCord solenoid}.  
\end{definition}

In \cite{McCord}, M. C. McCord proves that 
these solenoids are \emph{topologically homo\-geneous}, that is, for any two points $x,y \in \mathcal{X}$, there is a homeomorphism $f : \mathcal{X} \to \mathcal{X}$ such that $f(x) = y$. Then $\mathcal{X}$ is a connected principal fibre bundle with fibre a Cantor group and we will say $\mathcal{X}$ is a McCord solenoid based on $M$.
\medskip 

The reciprocal of McCord theorem for topologically homogeneous solenoidal manifolds was obtained by A. Clark and S. Hurder in the compact case. Namely, as proved in \cite{ClarkHurder}, any compact, connected, topologically homogeneous solenoidal manifold is homeomorphic to a McCord solenoid.
\medskip

As explained in Definition \ref{def:solenoidal_manifold}, all solenoidal manifolds considered here are assumed smooth. The inverse limit of any projective system of smooth coverings of a smooth manifold is endowed with the differentiable structure induced by the differentiable structure of the base. In general, these solenoids, regardless of whether they are McCord or not, may admit several inequivalent differentiable structures and other pathological phenomena (see \cite{VerjovskySullivan}). 
\medskip 

A McCord solenoid is always the suspension of a representation\linebreak
$\rho : \pi_1(M) \to G$ where $G$ is a Cantor group acting on itself by left translations (see \cite{VerjovskySullivan}).  In the forthcoming sections, reference to any McCord solenoid $\X$ will be usually interpreted as suspension of a representation $\rho : \pi_1(S) \to G$ of the fundamental group of a compact surface $S$ with a finite number of punctures in a Cantor group $G$ acting on itself by left translations. Furthermore we will assume that $\rho(\pi_1(S))$ is dense in $G$. It will therefore be a minimal hyperbolic solenoidal surface of finite type. 
\medskip 

The inverse limit of a projective system of finite coverings is the suspension of a representation $\rho : \pi_1(M) \to Homeo(K)$, where $K$ denotes the Cantor set. There is a distance on $K$ that is preserved by $\rho(\pi_1(M))$, so $\rho$ is in fact acting by isometries. However, without any further structure, the suspension of a representation $\rho:\pi_1(M)\to Homeo(K)$ is not, in general, an inverse limit of a projective system of finite coverings. 
\medskip 

The universal hyperbolic solenoid $\mathcal{H}^2$ and the solenoidal cusp $\mathcal{C}$ are McCord. The solenoidal surface in Example~\ref{example:Schottky} is not. To see this, notice that a McCord solenoid always admits a completely invariant measure of full support. The transverse dynamics in Example~\ref{example:Schottky} is that of a Schottky group on its limit set $\Lambda$, that contains hyperbolic and parabolic elements that do not preserve such a measure. Example~\ref{example:suspension_with_triadic_solenoid_over_closed_horocycle} is transversely isometric and then admits an unique completely invariant measure up to scaling \cite{Matsumotoequi}. We will soon see that it is not McCord. 

\begin{remark}
 An example of a compact solenoidal surface that is not homogeneous has been constructed by R. M. Schori in \cite{Schori}.
\end{remark}

\begin{remark}
In McCord solenoids, all leaves are homeomorphic. On the other hand, solenoidal surfaces obtained as the inverse limit of a family of coverings that are not regular can have leaves of different topological types. In fact, a single solenoidal surface can have dense leaves that include all possible homeomorphism classes of noncompact surfaces (see \cite{Alvarez-Brum-Martinez-Potrie}).
\end{remark}

\section{Hedlund's Theorem for hyperbolic solenoidal surfaces of finite type} \label{SHedlund}

In \cite{Hedlund}, G. H. Hedlund proved the following theorem:

\begin{namedtheorem}[Hedlund's] \label{thm:Hedlund} If $S$ is a hyperbolic surface of finite area, then all horocycles in $T^1 S$ are either periodic or dense.  
\end{namedtheorem}

Recall that $S$ is the quotient of the hyperbolic plane $\H$ by the action of a discrete group of isometries $\Gamma < \PSL$ having finite covolume. Therefore the unit tangent bundle $T^1 S$ identifies with the quotient $\Gamma \bs \PSL$ and horocycles in $T^1 S$ are the orbits of the natural right action $\PSL \curvearrowleft U$ where $U$ is the unipotent group given by
$$U = \left\{\, \matriz{1}{s}{0}{1} \mid \, s\in \R\, \right\}.$$

We aim to extend Hedlund's theorem to hyperbolic solenoidal surfaces of finite type that are obtained as the inverse limit of a tower of finite coverings. We will start by fixing notation. 

\medskip

Consider a non-compact hyperbolic surface of finite area, $S=\Gamma \bs \H$, where $\Gamma$ is a non-uniform lattice of $\PSL$ which is isomorphic to $\pi_1(S)$. A decreasing sequence of subgroups of finite index $\Gamma_n<\Gamma$ gives rise to a sequence of surfaces $S_n=\Gamma_n \bs \H$ and to a tower of coverings
$$\cdots S_{n+1} \overset{\pi_n}{\longrightarrow} S_n \longrightarrow \cdots \longrightarrow S_1 \overset{\pi_0}{\longrightarrow} S.$$
The surfaces $S_n$ are non-compact and have finite area. If $\Pi_n=\pi_{n-1}\circ\cdots\circ \pi_0$, then $\Pi_n:S_n\to S$ is a finite-to-one covering. For the moment, we will not need any additional assumptions on the groups $\Gamma_n$, in particular, we will not assume that they are normal subgroups of $\Gamma$, which is the condition for the inverse limit to be McCord. Denoting $\Gamma=\Gamma_0$, we have that $S=S_0$.
\medskip

The inverse limit of such a tower is the hyperbolic solenoidal surface of finite type 
$$\X=\varprojlim \, \big(S_{n+1}\overset{\pi_n}{\longrightarrow} S_n\big).$$
Its unit tangent bundle is 
$$\Y=T^1\X=\varprojlim \, \big(T^1S_{n+1}\overset{\pi_n}{\longrightarrow} T^1S_n\big),$$
and $\pi:\Y\to T^1 S$ is the natural projection. We are abusing notation by~calling $\pi_n$ both the covering $S_{n+1}\overset{\pi_n}{\longrightarrow} S_n$ and the covering $T^1S_{n+1}\overset{\pi_n}{\longrightarrow} T^1S_n$.
\medskip

As $S$ has finite volume, Hedlund's theorem describes the dynamics of the horocycle flow $h_{\R}$ on $T^1S$. Indeed, for any $y_0=\Gamma u\in T^1S$, we have the following dichotomy:
\smallskip 

\noindent
(i) The orbit $h_{\R}(y_0)$ is periodic, that is, $\overline{h_{\R}(y_0)} = h_{\R}(y_0)$.
 \smallskip 

\noindent
(ii) The orbit $h_{\R}(y_0)$ is dense, that is, $\overline{h_{\R}(y_0)} = T^1S$.
 \smallskip 

\noindent
In fact, according to \cite{Hedlund} (see also \cite{Eberlein}), the properties of the $h_{\R}$-orbit of $y_0 \in T^1S$ depend on the properties of the limit point $u(+\infty) \in \partial\H$ of the geodesic $u(t) \in \H$ determined by the unit tangent vector $u \in T^1 \H$: if $u(+\infty)$ is parabolic, the horocyclic trajectory is periodic and if $u(+\infty)$ is horocyclic, the orbit is dense. 

Our aim is to demonstrate the following extension, which has already been announced in the introduction: 

\begin{theorem}\label{thm:principal}
Under the previous hypotheses, for any point $y\in \Y$ with image $y_0=\pi(y)\in T^1S$, there are two possible cases:
 \smallskip  

\noindent
(i) The orbit $h_{\R}(y)$ projects onto a periodic orbit $h_{\R}(y_0)$,
  \smallskip  

\noindent
(ii) The orbit $h_{\R}(y)$ is dense, that is, $\overline{h_{\R}(y)} = \Y$.
\end{theorem}

\begin{proof}
 Let $y\in \Y$ be a point such that $\pi(y)=y_0$ is not periodic for the horocycle flow in $T^1S$. Recall that $y$ is a sequence
 $$y=(y_0,y_1,y_2,\ldots)\in \prod_{n\in\N}T^1S_n,$$
 and for all $n$, the point $y_n$ is not periodic for the horocycle flow in $T^1S_n$. 
 
 Let us take $x=(x_n)_{n\in \N}\in \Y$,   $\varepsilon>0$ and $N\in\N$. We define the set $U=U(x,\varepsilon,N)$ by
 $$U=\{x'=(x'_n)_{n\in\N}\in \mathcal{Y}\ :\ d_n(x_n,x'_n)<\varepsilon \ \forall n\leq N\},$$
 where $d_n$ is the Riemannian distance in $T^1S_n$. Sets of this form constitute a basis of the topology of $\Y$, so we have to prove that the horocycle orbit of $y$ intersects $U$. 
 
 In $T^1S_N$, the point $y_N$ has a dense horocycle orbit. Therefore, there exist a time $t\in\R$ such that $d_N(h_t(y_N),x_N)<\varepsilon$. Since for every $n<N$ the covering $\pi_n\circ\cdots\circ \pi_{N-1}:T^1S_N\to T^1S_n$ is 1-Lipschitz, we have that
 $$d_n(h_t(y_n),x_n)\leq \varepsilon$$
 for all $n\leq N$. That is, $h_t(y)\in U$.
\end{proof}

\begin{remark}
 In the case of a McCord solenoid, one can prove this result with arguments of a more ``homogeneous'' flavor. An alternative proof of Theorem~\ref{thm:principal} for McCord solenoids is the following:
\end{remark}

\begin{proof}
A McCord solenoid is the suspension of a representation $\rho : \Gamma \to G$ of the Fuchsian group $\Gamma$ as a dense subgroup of a Cantor group $G$. That is, $\Y$ is
the principal fibre bundle 
$$\Y=\Gamma \bs (\PSL\times G)$$ 
over the unit tangent bundle $T^1S=\Gamma \bs \PSL$, where the action of $\Gamma$ in $\PSL\times G$ is given by
$$\gamma(x,g)=(\gamma x, \rho(\gamma)g).$$
$\Y$ has finite ``volume'' with respect to the measure $\nu_{\Y}$ induced by the Haar measure of the product group $\PSL \times G$.

The measure $\nu_{\Y}$ is ergodic for the right $\PSL$-action on $\Y$. This follows from the fact that $\rho(\Gamma)$ is dense in $G$. In this setting, Moore's ergodicity theorem (see \cite[Chapter 2]{Zimmer}) implies that the horocycle flow on $\Y$ is also ergodic, and therefore transitive. If $x=\Gamma(u,g)\in \Y$ has dense horocycle orbit, so does $\Gamma(u,g')$ for any $g'\in G$, because $G$ acts by right translations in $\PSL\times G$ and commutes with both the $\Gamma$-action and the horocycle flow.   

Now, we will see that for any $y\in\Y$ such that $\pi(y)$ is not periodic in $T^1S$, there exists $g'\in G$ such that $x'=\Gamma (u,g')\in \overline{h_{\R}(y)}$. In fact, since $h_{\R}(\pi(y))$ is dense in $T^1S$ there exists a sequence of times $s_n$ such that $h_{s_n}(\pi(y))\to \pi(x)$. Up to chosing a subsequence, because $G$ is compact, $h_{s_n}(y)$ converges in $\Y$ to a point $x'\in \pi^{-1}(\pi(x))$, as claimed. 
\end{proof}

\bigbreak

Let $\Y$ be, as in Theorem~\ref{thm:principal}, the inverse limit of a tower of coverings over $T^1S$, where $S$ is a hyperbolic surface of finite area. The geodesic and horocycle flows in $\Y$ inherit many dynamical properties from the flows in the surfaces $T^1S_n$, $n\in \N$. Theorem~\ref{thm:principal} is an example of this phenomenon; the following proposition constitutes another one.

\begin{proposition}\label{proposition:topological_mixing}
  The geodesic and horocycle flows on $\Y$ are topologically mixing.
\end{proposition}

Recall that a flow $\varphi_t$ on $\Y$ is topologically mixing if, given open sets $U$ and $V$ in $\Y$, there exists $T\in \R$ such that for all $t>T$
\begin{equation}\label{eqn:topological_mixing}
\varphi_t(U)\cap V\neq \emptyset.
\end{equation}

\begin{proof}
For each $n\in\N$, let $p_n:\Y\to T^1S_n$ be the projection
\begin{eqnarray*}
p_n:\Y &\to& T^1S_n\\
x &\mapsto &x_n.
\end{eqnarray*}
A subbasis of the product topology on $\Y$ is given by sets of the form
$$p_n^{-1}(U_n),$$
as $n$ varies in $\N$ and $U_n$ is an open subset of $T^1S_n$.

Let $\varphi_t$ be either the geodesic or the horocycle flow in $\Y$. Let $\varphi^{(n)}_t$ be the corresponding (i.e. geodesic or horocycle) flow in $T^1S_n$, for all $n$. We will prove (\ref{eqn:topological_mixing}) for sets $U$ and $V$ of this form, namely
$$U=p_n^{-1}(U_n) \mbox{ and } V=p_m^{-1}(V_m).$$
Without loss of generality (by shrinking one of this sets if necessary) we will assume that $n=m$.

The flow $\varphi_t^{(n)}$ in $T^1S_n$ is topologically mixing, so there exists a $T$ such that for all $t>T$
$$\varphi^{(n)}_t(U_n)\cap V_n\neq \emptyset.$$
On the other hand, $\varphi^{(n)}_t\circ p_n=p_n\circ\varphi_t$.
Therefore, 
\begin{eqnarray*}
 \varphi_t(U)\cap V &=& \varphi_t(p_n^{-1}(U_n))\cap p_n^{-1}(V_n)\\
 &=& p_n^{-1}(\varphi^{(n)}_t(U_n))\cap p_n^{-1}(V_n)\\
 &=& p_n^{-1}(\varphi^{(n)}_t(U_n)\cap V_n)\neq \emptyset,
\end{eqnarray*}
as claimed. 
\end{proof}

\section{Minimal sets for the horocycle flow}
\label{Sminimalsets}

For a hyperbolic surface $S$, it is known from \cite{Dal'Bo} that all horocycles are dense if and only if the surface is compact. On the other hand, there exist non-compact hyperbolic surfaces for which the horocycle flow has no minimal sets, see \cite{MatsumotoMinimal}.

If $\X$ is a hyperbolic solenoidal surface of finite type and $\Y=T^1\X$ is its unit tangent bundle, the horocycle flow is never minimal in $\Y$. To see this, consider a fibration $\pi:\Y\to T^1S$ over a unit tangent bundle of a non-compact hyperbolic surface $S$ of finite area. If $H\subset T^1S$ is a closed horocyclic orbit, $\pi^{-1}(H)\subset \Y$ is a closed invariant proper subset. Furthermore, since $\pi^{-1}(H)$ is compact, this guarantees the existence of minimal sets for the horocycle flow in $\Y$.

\begin{definition} \label{def:minimalcusp}
We will say that a compact set $\mathcal{S}\subset \Y$ which is invariant under the horocycle flow in $\Y$ is a \emph{minimal set over a closed horocycle} if $\mathcal{S}$ is minimal for the horocycle flow and $\pi(\mathcal{S})$ is a closed horocyclic orbit in $T^1S$.
\end{definition}

In this section these minimal sets are described for hyperbolic solenoidal surfaces defined by towers of coverings. 

\subsection{A trichotomy for inverse limits of surface coverings} 

We will use the same notation as in the previous section where 
 $S=\Gamma \bs \H$ is a non-compact hyperbolic surface of finite area, and $\X$ is the inverse limit of a tower of coverings
$$\cdots S_{n+1} \overset{\pi_n}{\longrightarrow} S_n \longrightarrow \cdots \longrightarrow S_1 \overset{\pi_0}{\longrightarrow} S$$
corresponding to a decreasing sequence of subgroups $\Gamma_n<\Gamma$ of finite index.
Its unit tangent bundle is 
$$\Y=T^1\X=\varprojlim \, \big(T^1S_{n+1}\overset{\pi_n}{\longrightarrow} T^1S_n\big),$$
and $\pi:\Y\to T^1 S$ is the natural projection. 
\medskip

For every cusp in $T^1S=\Gamma\bs \PSL$ there is a one-parameter family of closed horocyclic orbits in $T^1S$. In $\Y$ there are horocyclic orbits $h_{\R}(y)$ which project onto closed orbits $h_{\R}(x)$, for $x\in T^1S$. This section aims to understand their closures. 
Together with Theorem~\ref{thm:principal}, this fully describes the minimal sets for the horocycle flow in $\Y$.
\medskip

These minimal sets can be modelled on some classical dynamical systems on a Cantor set:~\emph{odometers}. Odometers are defined as follows.

\begin{definition}\label{def:odometers}
Let $(F_n,T_n,q_n)_{n\in\N}$ be a sequence of finite sets $F_n$, bijective maps $T_n:F_n\to F_n$ and surjective maps $q_n:F_{n+1}\to F_n$ which satisfy $T_n\circ q_n=q_n\circ T_{n+1}$. When the cardinality of the $F_n$ is strictly increasing and we endow the $F_n$ with the discrete topology, the inverse limit 
$$Z= \varprojlim \, \big(q_n:F_{n+1}\to F_n\big)$$
is a Cantor set. The maps $T_n$ induce a homeomorphism $T:Z\to Z$, which is called an \emph{odometer}. 
\end{definition}

\begin{remark}
 It is a well-known and easy fact that $T$ is minimal if and only if for every $n$, the action of $T_n$ on $F_n$ is transitive. Namely, if $T_n$ is a permutation of $F_n$ consisting of a single cycle.
\end{remark}

If $\xi \in \partial \H$ is fixed by a parabolic element in $\Gamma$, it gives rise to a cusp in $S$. Under the covering $\Pi_n: S_n\to S$ this cusp can lift to exactly one cusp in $S_n$ or to several cusps in $S_n$, depending on how $\Gamma_n$ acts on the orbit $\Gamma\xi\subset \partial \H$.
We can now detail the statement of the second theorem announced in the introduction which completes the description given in Theorem~\ref{thm:principal}.

\begin{theorem}\label{thm:dynamics_over_the_cusps}

Let $\Gamma\subset \PSL$ be a non-uniform lattice, and $(\Gamma_n)_{n\in\N}$ be a decreasing sequence of finite index subgroups of $\Gamma$. This defines the inverse limit $\Y$ described above, and the projection $\pi:\Y\to T^1S$. 

Let $x \in T^1S$ be a periodic point for the horocycle flow. It is associated to a cusp in $S$ which we will call $\xi$. For each $n\in \N$, let $c_n$ be the number of cusps in $S_n$ which project via $\Pi_n$ onto the cusp $\xi$. (Notice that $c_n\leq c_{n+1}\ \forall n$.)
Let $\M=\pi^{-1}(h_{\R}(x))$. 
\smallskip

\noindent
(1)  If $c_n=1$ for all $n$, then $\M$ is a one-dimensional solenoid that is a minimal set for the horocycle flow on $\Y$.
\smallskip

\noindent
(2) If $\max_{n\in\N}c_n=m\in\N$, then $\M$ is the disjoint union of $m$ one-dimensional solenoids that are minimal sets for the horocycle flow on $\Y$.
\smallskip

\noindent
(3) If $\{c_n\}$ is unbounded,  $\M$ is the disjoint union of infinitely many minimal sets for the horocycle flow. 
\end{theorem}

\begin{proof} Assume without loss of generality that the vector $x$ is periodic of period one. Let $F_n=\Pi_n^{-1}(x)\subset T^1S_n$. (We are slightly abusing notation by using the same name $\Pi_n$ for the projections $S_n\to S$ and $T^1S_n\to T^1S$.) Then $F_n$ is a finite set of cardinal $[\Gamma:\Gamma_n]$. Let $T_n$ be the restriction to $F_n$ of the time-one map of the horocycle flow on $T^1S_n$.

If $c_n=1$, this means that the closed horocyclic orbit through $x$ in $T^1S$ is the image under $\Pi_n$ of a unique closed horocyclic orbit on $T^1S_n$, of period $[\Gamma:\Gamma_n]$. In this case $T_n$ acts transitively on $F_n$. If this is the case for all $n$, the limit map $T$ is a minimal odometer on the Cantor set $\pi^{-1}(x)$, and its suspension gives the 1-dimensional solenoid $\M$, which is foliated by dense horocyclic orbits.

If the maximum value of $c_n$ is $m\in \N$, there exists some $n_0$ such that $c_n=m$ for all $n\geq n_0$. Namely, $S_{n_0}$ has $m$ different cusps $\xi_1,\ldots,\xi_m$ that project onto $\xi$. This means that $T_{n_0}$ is a permutation of $F_{n_0}$ that decomposes into $m$ disjoint cycles. Reasoning as in the previous case with each of these cusps, we see that $\M$ is the disjoint union of $m$ 1-dimensional solenoids in $\Y$ which are minimal for the horocycle flow.

If the sequence $c_n$ is unbounded, let $x_1,\cdots,x_{c_n}\in T^1S_n$ be vectors such that $\Pi_n(x_i)=x$ for $i=1,\ldots,c_n$ and which correspond to different cusps in $S_n$. 
As in the proof of Proposition~\ref{proposition:topological_mixing}, let $p_n :\Y\to T^1S_n$ be the canonical projection (where $p_0=\pi$).
Then 
$$\M=\bigsqcup_{i=1}^{c_n}  p_i^{-1}\big(h_{\R}(x_i)\big)$$ 
is a partition of $\M$ into $c_n$ compact, disjoint one-dimensional solenoids which are invariant under the horocycle flow in $\Y$. This partition at level $n+1$ is a refinement of that at level $n$. Therefore, $\M$ is partitioned infinitely many minimal sets. 
\end{proof}

\subsection{Examples of McCord solenoids of different class}
 
In Theorem~\ref{thm:dynamics_over_the_cusps}, we have described three different situations according to the number of minimal sets projecting onto a closed horocycle around a cusp. Here we provide examples for each of these classes.

\begin{example}[A McCord solenoid of class (1)]
\label{example:class1}

Here we will give an example of McCord solenoid based on a surface $S=S_0$ with exactly two cusps. For each of them, $c_n = 1$ for all $n \in \N$.
Let $S_n=S_{2^n,2}$ be a non-compact hyperbolic surface of genus $2^n$ with 2 cusps. There is a 2-to-1 covering $\pi_n:S_{n+1}\to S_n$ for which each cusp of $S_{n}$ lifts to  a cusp of $S_{n+1}$ and for which every handle of $S_n$ lifts to two handles in $S_{n+1}.$ See Figure~\ref{figure:keep_cups}. Since these coverings are 2-to-1, they are regular. The limit
$$\X=\varprojlim \, \big(S_{n+1}\overset{\pi_n}{\longrightarrow} S_n\big)$$
has exactly one minimal set over each closed horocycle of $S$.

 \begin{figure}[h!]  
 \includegraphics[width=0.8\textwidth]{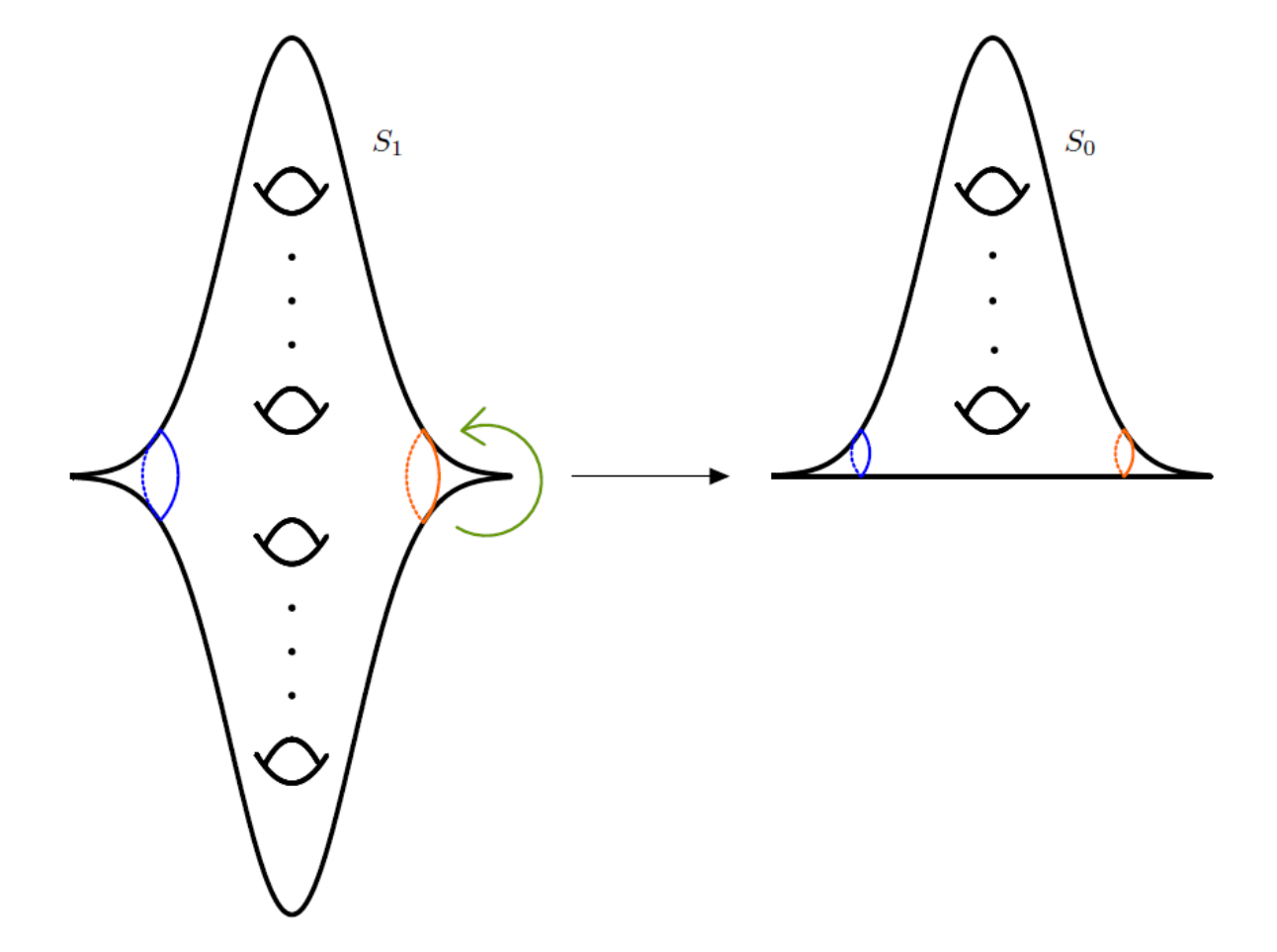}
 \caption{2-to-1 covering that keeps the number of cusps}
 \label{figure:keep_cups}
 \end{figure}

If $\Y=\varprojlim \, \big(T^1S_{n+1} \overset{\pi_n}{\longrightarrow} T^1S_n\big)$ is its unit tangent bundle, $\pi:\Y\to T^1S$ is the projection onto $T^1S$ and $H$ is a closed horocyclic orbit in $T^1S$, then $\pi^{-1}(H)$ is the dyadic solenoid $\mathcal{S}_2$. The dyadic solenoid is the suspension of the translation by 1 in the dyadic integers, or, equivalently, the inverse limit of a tower of 2-to-1 coverings of the circle.
\end{example}

\begin{example}[A McCord solenoid of class (1) over a surface with $m > 1$ cusps]
In the above example, the base surface $S_0$ had two cusps, as well as its coverings $S_n$. A similar example can be constructed starting from a surface $S_0$ with any number of cusps $m>1$. To see this, it is enough to show that if $S=S_{g,m}$ is a surface of genus $g\geq 0$ and $m$ cusps, there is a regular covering $p:S'\to S$ where $S'$ is also a surface with $m$ cusps.

The fundamental group of $S$ is 
$$
\Gamma= \langle \alpha_1,\beta_1,\ldots,\alpha_g,\beta_g,\gamma_1,\ldots,\gamma_m\, : \, \prod_{i=1}^g[\alpha_i,\beta_i]=\gamma_1\cdots\gamma_m \rangle.
$$
(Even if this argument is purely topological, we will think of $\Gamma$ as a Fuchsian group and of $\H$ as the universal covering of $S$.)
It is freely generated by $\{\alpha_1,\beta_1,\ldots,\alpha_g,\beta_g,\gamma_1,\ldots,\gamma_{m-1}\}$. Therefore, there is a group homomorphism $\varphi:\Gamma\to \Z/m\Z$ defined by 
\begin{eqnarray*}
 \varphi(\alpha_i)&=&0\mod m\ \mbox{ for } i=1,\ldots,g\\
 \varphi(\beta_i)&=&0\mod m \mbox{ for } i=1,\ldots,g\\
 \varphi(\gamma_j)&=&1\mod m \mbox{ for } j=1,\ldots,m-1.
\end{eqnarray*}
From the relation that defines $\Gamma$, it follows that $\varphi(\gamma_m)=1\mod m$.

Let $\Gamma'=\ker(\varphi)$, $S'=\Gamma'\bs\H$ and $p:\Gamma'\bs\H\to \Gamma\bs\H$. Then $p$ is a regular covering because $\Gamma'\lhd\Gamma$, and it has $m$ sheets. 

Finally, $S'$ has $m$ cusps. This is because $\varphi(\gamma_i)$ has order $m$ in $\Z/m\Z$ for all $i$,  so a closed horocycle in $S'$ will always project onto a closed horocycle in $S$ $m$-to-1. 

\label{example:class1m}

\end{example}

\begin{example}[A McCord solenoid of class (2)]
\label{example:class2}
A slight modification of the two previous constructions will yield McCord hyperbolic solenoidal surfaces with a finite number of minimal sets over each closed horocycle. 

Let $S_0=S_{1,1}$ be a hyperbolic surface of genus one and one cusp. Its fundamental group is $\Gamma_0=\langle\alpha, \beta\rangle$, where the commutator $\gamma=[\alpha,\beta]$ represents a curve that turns once around the cusp. There is a 2-to-1 covering $\pi_0:S_1\to S_0$ corresponding to the subgroup $\Gamma_1 \lhd \Gamma_0$ that is (freely) generated by $\alpha^2$, $\beta$ and $\gamma$. The surface $S_1$ has genus one and two cusps. See Figure~\ref{figure:duplication_of_cusps}. The covering is regular because $\Gamma_1$, having index 2, is normal in $\Gamma_0$.

 \begin{figure}[h!]  
 \includegraphics[width=0.9\textwidth]{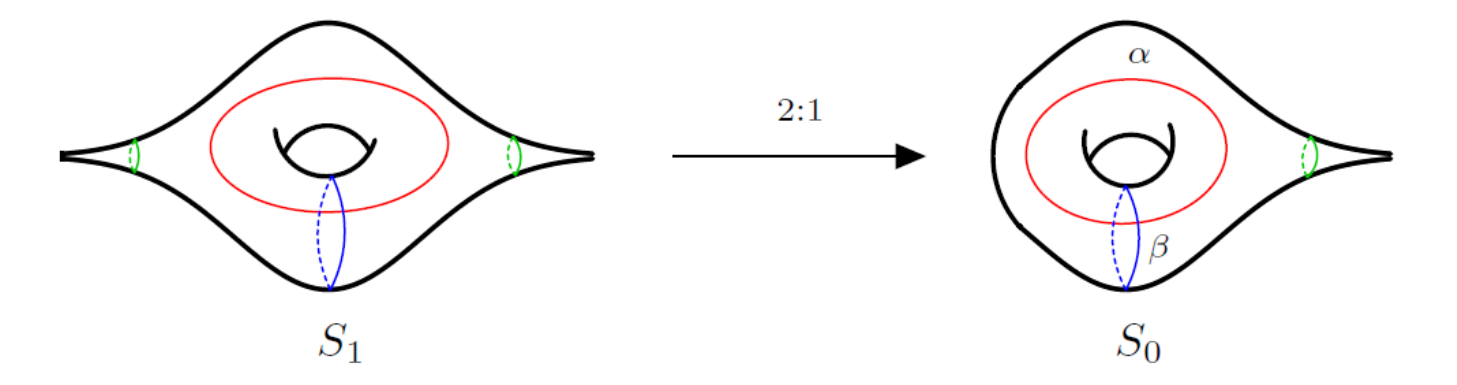}
 \caption{2-to-1 covering that duplicates the number of cusps}
 \label{figure:duplication_of_cusps}
 \end{figure}

Now composing with the 2-to-1 coverings described in Example~\ref{example:class1} above, we obtain a sequence of finite-to-one coverings $\Pi_n:S_n \to S_0$ such that $c_n = 2$ for all $n$.

To obtain an example based on the same surface $S_0 = S_{1,1}$ with $m$ minimal sets over each closed horocycle of $S_0$, it is enough to replace the 2-to-1 covering from $S_1 = S_{1,2}$ to $S_0$ with an $m$-to-1 (cyclic) covering from $S_1 = S_{1,m}$ to $S_0$ and the 2-to-1 coverings from Example~\ref{example:class1} with the $m$-to-1 coverings from Example~\ref{example:class1m}.

 \end{example}

\begin{example}[A McCord solenoid of class (3) with infinitely minimal sets which are closed orbits]
\label{example:class3closed}
 Let $S_0=S_{1,1}$ and $\pi_0:S_1\to S_0$ be the surface and the 2-to-1 covering of the previous Example~\ref{example:class2}. See again Figure~\ref{figure:duplication_of_cusps}. 
Similarly, $S_1$ has a 2-to-1 covering that has genus one and 4 cusps, which has a 2-to-one covering of genus one and 8 cusps, etc. There is a tower of 2-to-1 coverings $\pi_n:S_{n+1}\to S_n$, where $S_n$ has $2^n$ cusps for all $n$. 
 
 The inverse limit
 $$\X=\varprojlim \, \big(S_{n+1} \overset{\pi_n}{\longrightarrow} S_n\big)$$
 is a hyperbolic solenoidal surface of finite type, over the base $S_0$ which has one cusp, for which $c_n\to\infty$. 
 
 When we take the coverings of the unit tangent bundles, each closed horocycle of length one in $T^1S_n$ lifts, via $\pi_n$, to 2 closed horocycles of length one. 
 For the inverse limit
 $$\Y=\varprojlim \, \big (T^1S_{n+1} \overset{\pi_n}{\longrightarrow} T^1S_n \big)$$
 this means that the minimal sets for the horocycle flow in $\Y$ are closed horocyclic orbits.
\end{example}

To get a McCord solenoidal surface of class (3)  for which the minimal sets for the horocycle flow are not closed orbits, we need a sequence $(\Gamma_n)_{n\in\N}$ of non-uniform lattices of $\PSL$ such that 
\begin{itemize}
 \item $\Gamma_{n+1}\lhd \Gamma_n$ for all $n$,
\item if $c_n$ is the number of cusps of $S_n=\Gamma_n\bs \H$, then $\lim_{n\to\infty}c_n=\infty$, and
 \item if $[\Gamma_n:\Gamma_{n+1}]=d_n$  then $c_{n+1}<d_nc_n$ for an infinite number of $n\in\N$.
 \end{itemize}
 
 The next example satisfies these three conditions. 

\begin{example}[A McCord solenoid of class (3) with minimal sets which are 1-dimensional solenoidal manifolds]\label{ex:arithmetic}

For $n\geq 2$, we consider the projection $q_n:\SLZ \to {\rm SL}2,\Z / n\Z)$ given by

\begin{equation*}
  q_n\left(\begin{array}{cc} a&b\\ c&d \end{array}\right)=\left(\begin{array}{cc} a\mod n & b \mod n\\ c \mod  n&d \mod n \end{array}\right).
\end{equation*}
It is an epimorphism of groups. This fact is not immediately obvious but it is well known.

\begin{definition}
 The group $\Gamma(n) = \ker(q_n)$ is called the \emph{principal congruence subgroup} of level $n$.
\end{definition}

Let $\Gamma_n$ be the subgroup of $\PSLZ$ obtained as the  image of $\Gamma(n)$ under the projection $\SLZ \to \PSLZ$. For $n=1,2$ the projection $\Gamma(n) \to \Gamma_n$ is 2-to-1, and for $n\geq 3$ it is injective. 

Notice that $\Gamma_1=\PSLZ$ and $\Gamma_1\bs \H$ is not a surface but an orbifold. Nevertheless, $\Gamma_n$ is torsion-free for $n\geq 3$. Much is known about the $\Gamma_n$ and the \emph{modular surfaces} $S_n=\Gamma_n\bs \H$. For example, $\Gamma_5\bs \H$ is a surface of genus 0 and 20 cusps. The surface $\Gamma_7\bs \H$, is the so-called the \emph{Klein cubic}, and has genus 3.

We will use the following well-known facts about $\Gamma_n$ (see \cite{Shimura1971}):
\medskip 

\noindent
(1) The subgroup $\Gamma_n < \PSLZ$ has index 
\begin{align*} 
[\PSLZ:\Gamma_n] & =\frac{1}{2}n^3\prod_{\small{\begin{array}{cc} p \mbox{ prime }\\ p| n \end{array}}} \left(1-\frac{1}{p^2}\right) \\
 & = \frac{1}{2}n^2\varphi(n)\prod_{\small{\begin{array}{cc} p \mbox{ prime }\\ p| n \end{array}}} \left(1+\frac{1}{p}\right),
 \end{align*}
 where $\varphi(n)$ is Euler's totient function.
\smallskip 

\noindent
(2) The surface $\Gamma_n\bs \H$ has 
 \begin{align*}
  c_n &= \frac{1}{2}n^2\prod_{\small{\begin{array}{cc} p \mbox{ prime }\\ p| n \end{array}}} \left(1-\frac{1}{p^2}\right) = \frac{1}{2}n\varphi(n)\prod_{\small{\begin{array}{cc} p \mbox{ prime }\\ p| n \end{array}}} \left(1+\frac{1}{p}\right)
 \end{align*}
cusps. 
\medskip 

Notice that $\Gamma_n$ is a subgroup of $\Gamma_m$ if and only if $m$ divides $n$.
In order to have a tower of coverings, consider $S_n=\Gamma_{n!}$ and
$$
\pi_n:S_{n+1}\to S_{n}.
$$

Let $\X$ be its inverse limit, and $\Y$ the unit tangent bundle of $\X$. Condition (2) says that $c_n\to \infty$, so $\Y$ has infinitely many (in fact, uncountably many) minimal sets over each closed horocycle of $S_1$. Together, Conditions (1) and (2) imply that, under the projection $\Pi_n=\pi_{n-1}\circ\cdots\circ \pi_1:S_n\to S_1$, a closed horocycle in $S_n$ projects onto a closed horocycle in $S_1$ in a $n!$-to-1 way. Equivalently, a closed horocycle in $S_{n+1}$ projects onto a closed horocycle in $S_n$ in a $(n+1)$-to-1 way. Therefore, the minimal sets for the horocycle flow in $\Y$ are 1-dimensional solenoidal manifolds.
\end{example}

Finally, to complete the description of three different classes provided by Theorem~\ref{thm:dynamics_over_the_cusps}, we will add two remarks: 
\medskip

\noindent 
(I) There are no McCord solenoidal surfaces of class (1) based on hyperbolic surfaces with one single cusp:

\begin{proposition}\label{proposition:no_regular_covering_with_one_cusp}
Consider a tower of hyperbolic surfaces of finite area whose base $S$ has exactly one cusp. With the notations of Theorem~\ref{thm:dynamics_over_the_cusps}, 
if in addition $\Gamma_{n+1}\lhd \Gamma_n$ for all $n$, that is, if $\X$ is McCord, 
it always has more than one minimal set over any closed horocycle.
\end{proposition}

\begin{proof}
 It is enough to show that the surface $S=S_{g,1}=\Gamma_0\bs\H$ of genus $g$ and one cusp does not have a normal cover with one cusp. Suppose that it does, that is, that there is a subgroup $\Gamma_1\lhd \Gamma_0$ of finite index $m$ such that $S'=\Gamma_1\bs \H$ has one cusp.
 
 The standard generators for $\Gamma_0\simeq \pi_1(S)$ are $\alpha_1,\beta_1,\ldots,\alpha_g,\beta_g$. The group $\Gamma_0$ is free in these generators, and the product of commutators $\gamma=\prod_{i=1}^g[\alpha_i,\beta_i]$ represents a curve that turns once around the cusp. 
 
 Since $S_1$ has only one cusp, it has to project $m$-to-1 onto the cusp in $S$. This means that, in the quotient $\Gamma_1 \bs \Gamma_0$, the element $\Gamma_1 \gamma$ has order $m$. Since $m$ is exactly the order of $\Gamma_1 \bs \Gamma_0$, this quotient group is the cyclic group generated by $\Gamma_1 \gamma$, and it is therefore Abelian. But since $\gamma$ is a product of commutators, it has to project trivially to any Abelian quotient of $\Gamma_0$, which is a contradiction.
\end{proof}

\noindent
(II) On the contrary, according to \cite{Edmonds1984}, the only obstruction to the existence of a non-regular covering of hyperbolic punctured surfaces is provided by the Riemann-Hurwitz formula. In particular, there is a tower of 3-to-1 coverings $\pi_n:S_{n+1}\to S_n$, where $S_0=S_{1,1}$ and $S_n$ has genus $(3^n+1)/2$ and only one cusp for all $n$. The inverse limit
$$\X=\varprojlim \, \big(S_{n+1} \overset{\pi_n}{\longrightarrow} S_n \big)$$
is a solenoidal surface based on the once-punctured surface $S = S_{1,1}$ with exactly one minimal set over any closed horocycle of $S$, but it is not McCord by Proposition~\ref{proposition:no_regular_covering_with_one_cusp}.
Thus all possible situations described in Theorem~\ref{thm:dynamics_over_the_cusps} are realizable by inverse limits of finite coverings.
\medskip 

Combining both remarks we deduce that the solenoidal surface constructed in  Example~\ref{example:suspension_with_triadic_solenoid_over_closed_horocycle}   can not be McCord, but it is the inverse limit of a tower of non-regular 3-to-1 coverings. 

\section{Minimal sets over closed horocycles with small and large period} \label{Slargeperiods}

For any hyperbolic solenoidal surface $\X$, minimal sets over closed horocycles arise in one-parameter families, since if $\mathcal{S}$ is a minimal set over a closed horocycle, so is $g_t(\mathcal{S})$ for every $t\in \R$.

In this section we will begin by describing the set of these families of minimal sets over cusps, which amounts to describing the set of ends of $\Y$. Later, we will see how each family moves inside $\Y$ as $t\to\pm\infty$ assuming $\Y$ is McCord. 

This behaviour is exactly what is observed in non-compact hyperbolic surfaces of finite area. As we already recalled in the introduction, the phenomenon is actually stronger. According to a theorem by P. Sarnak \cite{Sarnak1981}, if $\nu$ is an ergodic probability measure in $T^1S$ which is supported in a closed horocyclic orbit, then the weak* limit $\lim_{t\to -\infty} (g_t)_*\nu$ exists and it is equal to the Liouville measure in $T^1S$ (up to scaling). This implies that, in the end compactification $\widehat{T^1S}$ of $T^1S$, closed horocycles converge to the whole $\widehat{T^1S}$ when pushed by the geodesic flow for negative times.

\subsection{Cuspidal ends of $\Y$}

A hyperbolic surface of finite area $S$ --as well as its unit tangent bundle $T^1S$-- admits a natural compactification, where each cusp is compactified with one point. It is nothing but the \emph{end compactification} (in the sense of H. Freudenthal) $\widehat S$ of $S$ --and similarly $\widehat{T^1S}$ of $T^1S$--. Namely, for any increasing sequence of compact subsets 
$$
K_1 \subset K_2  \subset  K_n \subset K_{n+1} \subset \dots
$$ 
whose interiors cover $S$, each decreasing sequence 
$$
U_1 \supset U_2 \supset \dots U_n \supset U_{n+1} \supset \dots $$  
of connected components $U_n$ of the complements $X - K_n$ becomes a fundamental system of neighbourhoods of an end. As $S$ --and also $T^1S$-- is locally compact and $\sigma$-compact, we can equivalently define the end compactification of $S$ as follows: for any direct family of compact subsets $K$ and inclusion maps $K \hookrightarrow K'$ with direct limit $S$, the space of ends is the inverse limit of the family of groups $\pi_0(S-K)$ and homomorphisms $\pi_0(S -K') \to \pi_0(S-K)$ induced by the inclusions $S-K' \hookrightarrow X-K$. In fact, as $S$ and $T^1S$ are manifolds, we can replace each compact $K_n$ by a compact submanifold with boundary and each connected component $U_n$ by an unbounded connected submanifold. 

We will see that something similar is true for $\Y$.
To this effect, we will consider the tower of coverings 
$$\cdots \longrightarrow S_{n+1} \overset{\pi_n}{\longrightarrow} S_n \longrightarrow \cdots \longrightarrow S=S_0$$
and the corresponding tower for the unit tangent bundles
$$\cdots \longrightarrow \Gamma_{n+1}\backslash \PSL \overset{\pi_n}{\longrightarrow} \Gamma_n \backslash \PSL \longrightarrow \cdots \longrightarrow \Gamma \backslash \PSL.$$
We are (again!) abusing notation by calling $\pi_n$ both the surface coverings and its tangent maps. 
In this subsection we are not assuming that the subgroups $\Gamma_n$ are normal --that is, that the inverse limit is McCord-- since this is not necessary for the main result, Proposition~\ref{proposition:compactification} below. 

Each surface $S_n$, as well as its unit tangent bundle $T^1S_n=\Gamma_n\backslash \PSL$, can be compactified using a finite number of points, one for each cusp. Let $\widehat{S_n}$ and $\widehat{T^1S_n}$ be these compactifications. Then the above coverings extend to ramified coverings 
$$\widehat \pi_n :\widehat{S_{n+1}}\to \widehat{S_n} \quad \text{and} \quad  \widehat \pi_n :\widehat{T^1S_{n+1}}\to \widehat{T^1S_n}.$$
The inverse limit
$$\widehat{\Y}=\varprojlim (\widehat{T^1S_{n+1}}\overset{\widehat \pi_n}{\longrightarrow} \widehat{T^1S_n})$$
is a compact space containing $\Y$ as a dense subset.

This compactification can be understood as follows. Consider the infinite sequences
$$(\xi_0, \xi_1,\ldots)\in \prod_{n=0}^\infty \widehat{T^1S_n}-T^1S_n$$
such that $\widehat \pi_n(\xi_{n+1})=\xi_n$. Each sequence corresponds to a choice of one cusp in each $S_n$, each of which projects onto the previous one under the covering map. These sequences are exactly the points in $\widehat \Y - \Y$. They can be identified with the boundary at infinity of a forest $\mathcal{T}$ that has vertices 
$$V=\bigsqcup_{n=0}^\infty \widehat{T^1S_n}-T^1S_n$$
and edges joining each $\widehat{T^1S_n}-T^1S_n$ with $\widehat{T^1S_{n+1}}-T^1S_{n+1}$ in such a way that there is an edge between $\xi_n$ and $\xi_{n+1}$ if and only if $\widehat \pi_n(\xi_{n+1})=\xi_n$. This forest can be replaced by a tree by adding a single vertex at level $n=-1$ and edges from this vertex to the vertices in $\widehat{T^1S}-T^1S$. When we say that $\widehat{\Y} -\Y$ identifies with $\partial \mathcal{T}$, what we mean is that the natural bijection between these two sets is a homeomorphism, if we consider $\widehat{\Y} - \Y$ as a topological subspace of $\widehat{\Y}$.

\begin{proposition} \label{proposition:decomposition}
The boundary at infinity $\widehat{\Y} - \Y$ of $\Y$ is a $0$-dimensional space that admits a natural decomposition as finite disjoint union of finite sets and Cantor sets.
\end{proposition} 

\begin{proof}  
Consider a cusp $\xi$ in $S$. It gives rise to a rooted tree $\T_\xi \subset \T$, as explained before. 
For each $n \in \N$, all vertices of the tree $\T_\xi$ at distance $n$ from the root have finite degree. 
This implies that the boundary at infinity of $\T_\xi$, that is a closed subset of $\widehat{\Y} - \Y$, decomposes as finite  disjoint union of finite sets or Cantor sets.
Since $S$ has finitely many cusps, the same happens for the boundary $\widehat{\Y} - \Y$.
\end{proof}

Informally, we could say that the compactification of $\Y$ is obtained by adding one point at infinity for each choice of pre-orbit of a cusp of $S$ in the tower of coverings. More formally,  the forest $\mathcal{T}$ can be interpreted as the union of the Bratteli diagrams (see \cite{GPS1995} for a definition) corresponding to the odometers $(F_n,T_n,q_n)_{n \in \N}$ associated with the cusps $\xi_0$ of $S_0$ according to the construction used in the proof of Theorem~\ref{thm:dynamics_over_the_cusps}. The same idea is applied to the boundary at infinity $\widehat{\Y} - \Y$ of the unit tangent bundle $\Y$ to prove the following result: 

\begin{proposition}\label{proposition:compactification}
Each point of $\widehat{\Y}-\Y$ corresponds to  a minimal set over a closed horocycle of period 1 or equivalently to a one-parameter family of minimal sets over a cusp.
\end{proposition}

\begin{proof} Let $\xi = (\xi_0,\xi_1,\ldots)\in \widehat{\Y}-\Y$.

Consider a sequence $x=(x_0, x_1,\ldots) \in \Y\subset \prod_{n=0}^\infty T^1S_n$ such that:
\begin{itemize}
 \item [\textbullet] $x_0\in T^1S$ is a periodic point of period 1 for the horocycle flow.
 \item [\textbullet] For all $n$, $\pi_n(x_{n+1})=x_n$.
 \item [\textbullet] For all $n$, the closed horocyclic orbit through $x_n$ in $T^1S_n$ is freely homotopic to the point $\xi_n$ in $\widehat{T^1S_n}$. 
\end{itemize}

We will refer to the last condition by saying that ``$x_n$ corresponds to the cusp $\xi_n$''. Notice that, given such a sequence $x$, the sequence $\xi$ can be recovered. Therefore, we can write $\xi =\xi (x)$. The converse is not true, since for a given $\xi$ there are infinitely many choices of $x$.

Let $\mathcal{M}_x=\overline{h_{\R}(x)}$. Any minimal set over a cusp must be of this type, since it is the closure of some horocyclic orbit. We will prove that $\mathcal{M}_x$ is minimal, and that $\mathcal{M}_{x'}=\mathcal{M}_x$ if and only if $\xi(x)=\xi(x')$.

To see that $\mathcal{M}_x$ is minimal, we will see that it is the suspension of a minimal odometer.  Although it is not exactly the odometer described in the proof of  Theorem~\ref{thm:dynamics_over_the_cusps}, we will keep the same notation. For every $n$, let 
$$F_n=\{\, h_k(x_n)\, : \, k\in\Z \, \},$$
that is, $F_n$ is the set of images of $x_n$ under integer times of the horocycle flow in $T^1S_n$. It is a finite set, since $x_n$ is periodic for $h_s$ of integer period. Let $T_n:F_n\to F_n$ be given by $T_n(x)=h_1(x)$. It is a permutation of $F_n$ which acts transitively. Let $q_n:F_{n+1}\to F_n$ be the restriction of $\pi_n$ to $F_{n+1}$. Then $T_n\circ q_n=q_n\circ T_{n+1}$. The sequence $(F_n,T_n,q_n)_{n\in \N}$ induces a minimal transformation $T$ on a closed subset $K_x$ of the Cantor set $\pi^{-1}(x_0)$. Then $K_x$ is invariant under the time-one map $h_1$ of the horocycle flow, which coincides with $T$. Saturating $K_x$ by the horocycle flow gives the suspension of $T$, so the horocycle flow in $\mathcal{M}_x$ is the suspension of $T$ --and it is minimal. Notice that if the cardinality of the $F_n$ remains bounded, $\mathcal{M}_x$ is a closed horocyclic orbit, and if the cardinality of the $F_n$ goes to infinity $\mathcal{M}_x$ is a minimal one-dimensional solenoid.
 
 If $x'\in \Y$ is another sequence such that $\xi(x)=\xi(x')$ with $x_0=x'_0$, then in fact $x'_n\in F_n$ for all $n$.  If $\xi(x)=\xi(x')$ and $x_0\neq x'_0$, there exists an $s\in(0,1)$ such that $x'_0=h_s(x_0)$. In any case, $\mathcal{M}_{x'}=\mathcal{M}_x$.
 
 If $\xi,\ \xi'\in \widehat{\Y}-\Y$ are different, this means that there exists an $n$ such that $\xi_n\neq \xi'_n$. Consider $x,x'\in \Y$ as above which correspond to $\xi$ and $\xi'$, respectively. Then $\pi^{-1}(h_{\R}(x_n))$ and $\pi^{-1}(h_{\R}(x'_n))$
are two compact disjoint invariant sets for the horocycle flow in $\Y$ such that $\mathcal{M}_x\subset \pi^{-1}(h_{\R}(x_n))$ and $\mathcal{M}_{x'}\subset \pi^{-1}(h_{\R}(x'_n))$. This proves that $\mathcal{M}_x\neq \mathcal{M}_{x'}$.
\end{proof}

Let $\mathcal{M}\subset\Y$ be a minimal set over a closed horocycle of period 1, and $\mathcal{M}_{\X}$ be its projection to $\X$.
From the proof of Proposition~\ref{proposition:compactification}, we see that $\mathcal{M}_{\X}$ bounds a connected solenoidal manifold that projects over the neighbourhood of the end of $S$ bounded by the closed horocycle in $S$. We deduce that $\widehat{\X}$ and $\widehat{\Y}$ coincide with the end compactifications of $\X$ and  $\Y$ respectively, and $\widehat{\Y} - \Y$ is the space of ends of $\Y$. This justifies the following definition: 

\begin{definition}
A point of $\widehat{\Y}-\Y$ is called a \emph{cuspidal end} of $\Y$.
\end{definition}

\begin{remark}
If $x\in \Y$ projects onto a periodic point of the horocycle flow on $T^1S$, we have just proved that its orbit closure $\M_x$ is a minimal set for the horocycle flow. In particular, for any $x_0\in T^1S$ which is periodic for the horocycle flow on $S$, the set of periodic orbits for the horocycle flow in $\M=\pi^{-1}(h_{\R}(x_0))$ is a clopen subset of $\M$ according to Proposition~\ref{proposition:decomposition}. This implies that any closed horocyclic orbit must have trivial holonomy.
Recall that Example \ref{example:dynamics_over_cusp_can_be_anything} can be chosen so that it has closed horocyclic orbits with non-trivial holonomy, by taking an homeomorphism $T:K\to K$ with periodic points, but for which periodic points do not form a clopen set. This cannot happen in the inverse limit of a tower of coverings.
\end{remark}

\subsection{A convergence theorem for McCord solenoidal surfaces}
Now let us assume that the minimal hyperbolic solenoidal surface of finite type $\X$ is McCord. When we describe it as a limit of a tower of coverings, this means that the subgroups $\Gamma_n$ are normal in $\Gamma$. As before, $\Y=T^1\X$. 

Theorem~\ref{thm:principal} tells us that horocyclic orbits either project onto closed horocyclic orbits in $T^1S$ or are dense. This means that any closed minimal set for the horocycle flow in $\Y$ is a minimal set over a closed horocycle, and in particular it is compact. Minimal sets are therefore grouped into one-parameter families, each corresponding to a end in $\widehat{\Y}-\Y$. 

Given a minimal set $\M$ for the horocycle flow in $\Y$, we will now discuss what happens with its one-parameter family $g_t(\M)$ when $t$ goes to $+\infty$ and to $-\infty$. We will prove that when $t\to +\infty$ $g_t(\M)$ goes to a cuspidal point, and when $t\to -\infty$ $g_t(\M)$ converges to ``the whole of $\Y$'', in an appropriate sense. 

Let us recall the precise statement of this result given in the introduction. The geodesic and horocycle flows in $\Y$ extend continuously to $\widehat \Y$ by leaving the cuspidal ends fixed. Consider a distance $d$ which induces the topology on $\widehat \Y$, and the corresponding Hausdorff distance on the compact subsets of $\widehat \Y$. 

\begin{theorem}\label{thm:convergence_in_Hausdorff_distance}
 Let $\X$ be a McCord hyperbolic solenoidal surface of finite type, $\Y=T^1\X$ its unit tangent bundle and $\widehat \Y$  the end compactification of $\Y$. If $\M\subset \Y$ is a minimal set for the horocycle flow, then
\smallskip

\noindent
(i) $\lim_{t\to +\infty}g_t(\M)$ is a point in $\widehat{\Y}-\Y$,
\smallskip

\noindent
(ii) $\lim_{t\to -\infty}g_t(\M)=\widehat{\Y}$.
 \smallskip

\noindent
These limits are taken in the Hausdorff distance.
\end{theorem}

From the construction of $\hat \Y$, it is clear that the first condition remains valid for any minimal set over a closed horocyclic orbit in the unit tangent bundle of any inverse limit of surface coverings. We will prove the second condition of  Theorem~\ref{thm:convergence_in_Hausdorff_distance} in several steps.
\medskip

There is a natural distance in the non-compact solenoidal manifold $\Y$, which does not extend to the compactification. 

\begin{definition}
 Let $d'$ be a distance in $\PSL$ coming from a left-invariant metric, and $d''$ a  bi-invariant distance in $G$. They define a distance $d$ in $\PSL \times G$ by 
$$d((u_0,g_0),(u_1,g_1))=d'(u_0,u_1)+d''(g_0,g_1).$$ 
The distance $\mathbf{d}$ in $\Y$ given by 
 $$
\mathbf{d}(x,y)=\min\{\, d((u_0,g_0),(u_1,g_1)) \, : \, \pi(u_0,g_0)=x,\, \pi(u_1,g_1)=y \,\}
$$ 
is called a \emph{natural distance}. It is not unique because it depends on the choices of $d'$ and $d''$.
\end{definition}

This distance has many good properties, for example that the inclusion of a leaf $T^1L\hookrightarrow \Y$ and the projection $\Pi_{\Gamma}:\PSL\times G\to \Y$ are both 1-Lipschitz. As $(\Y,\mathbf{d})$ has infinite diameter, we will consider compact subsets of $\Y$.

Recall that if $(X,d)$ is a metric space, $A\subset X$ and $\varepsilon>0$, $A$ is $\varepsilon$-dense if $\forall x\in X\ \ \exists y\in A$ such that $d(x,y)<\varepsilon$. More generally, we have the following definition:

\begin{definition}
 Let $(X,d)$ be a metric space and $A, K\subset X$. For $\varepsilon>0$ we say that $A$ is \emph{$\varepsilon$-dense in $K$} if $\forall x\in K\ \ \exists y\in A$ such that $d(x,y)<\varepsilon$.
\end{definition}

\begin{remark}
 This is not equivalent to saying that $A\cap K$ is $\varepsilon$-dense as a topological subspace of $K$. For example, if $A$ is dense in $X$ but $A\neq X$, $A$ is $\varepsilon$-dense in $X-A$ for every $\varepsilon$, but $A\cap (X-A)=\emptyset$. 
\end{remark}

The next Lemma is a version of Theorem~\ref{thm:convergence_in_Hausdorff_distance} stated in the non-compact space $\Y$. It is useful because it allows us to work with a natural distance $\mathbf{d}$ in $\Y$. It is the main result in this section, in the sense that Theorem~\ref{thm:convergence_in_Hausdorff_distance} follows easily from it.

\begin{lemma}\label{lemma:epsilon_densidad_en_K}
 Let $\Y$ be a minimal McCord solenoid obtained as inverse limit 
 $$
\Y=\varprojlim \big( \Gamma_{n+1}\backslash \PSL \longrightarrow \Gamma_n \backslash \PSL\big),
$$
endowed with the natural distance $\mathbf{d}$. (We still use the notation $\Gamma_0=\Gamma$.)
Then, for any $\varepsilon>0$, any compact subset $K$ of $\Gamma\backslash \PSL$ which contains a ball of radius greater than $\varepsilon$, and any minimal $h_s$-invariant set $\M\subset \Y$, there exists $T\in \R$ such that, for all $t<T$, $g_t(\M)$ is $\varepsilon$-dense in $\K=\pi^{-1}(K)$.
 \end{lemma}

Let us first see that the analogue of this lemma is known in the case of a surface, since our result relies on this fact.

\begin{remark}
As a consequence of Sarnak's theorem \cite{Sarnak1981}, the following is true. \emph{Let $S$ be a non-compact hyperbolic surface of finite area, $K$ a compact subset of $T^1S$ and $H$ a closed horocyclic orbit in $T^1S$. Then, for any $\varepsilon>0$, there exists $T\in\R$ such that for all $t<T$ the closed horocyclic orbit $g_t(H)$ is $\varepsilon$-dense in $K$.} However, Sarnak's powerful result does not seem to be essential and there should be an alternative proof of Theorem~\ref{thm:3} using only topological properties of the geodesic and horocycle flows.
\medskip

 In order to see this, let $\nu$ be the $h_s$-invariant probability measure supported on $H$. Namely, $\nu$ is the normalised length measure on $H$. The family of balls $\{B(x,\varepsilon/2)\}_{x\in K}$ covers $K$, so we can consider a finite subfamily $F=\{ B(x_i,\varepsilon/2)\}_{i=1}^n$ which still covers $K$. For each $i$, let $f_i:T^1S\to \R$ be a continuous function with support in $B(x_i,\varepsilon/2)$ such that $\int_{T^1S} f_i(x)\, dx=1$. This integral is taken with respect to the Liouville measure on $T^1S.$

 According to Sarnak's theorem, for all bounded continuous functions $f:T^1S\to \R$, the limit
 $$\lim_{t\to -\infty} \int_{T^1S}f(x)\ d(g_t)_*\nu(x) = \int_{T^1S} f(x)\, dx,$$
so there exists $T$ such that
 $$\int_{T^1S} f_i(x)\, d(g_t)_*\nu(x) >1/2$$
for all $t<T$ and $i=1,\ldots,n$. 
Since the support of the measure $(g_t)_*\nu$ is the closed horocyclic orbit $g_t(H)$, this orbit must intersect all elements of $F$ for $t<T$. Let $x\in K$. Since $F$
covers $K$, $x\in B(x_i,\varepsilon/2)$ for some $i$. For this $i$, $B(x_i,\varepsilon/2)\subset B(x,\varepsilon)$. This implies that for $t<T$, $g_t(H)\cap B(x,\varepsilon)\neq \emptyset$. This means that, for $t<T$, $g_t(H)$ is $\varepsilon$-dense in $K$, as claimed.
\end{remark}

 We will need a couple of technical lemmas in order to prove Lemma~\ref{lemma:epsilon_densidad_en_K}.
 
 First, recall that the McCord solenoid $\Y=\Gamma \backslash (\PSL\times G)$ is obtained as the suspension of a representation $\rho:\Gamma\to G$ which has dense image. This gives rise to a commutative diagram
 \[ \begin{tikzcd}
\PSL\times G \arrow{r}{\Pi_{\Gamma}} \arrow[swap]{d}{\pi_1} & \Y \arrow{d}{\pi} \\%
\PSL \arrow{r}{\pi_{\Gamma}}& T^1S
\end{tikzcd}
\]
 where $\pi_1$ is the canonical projection onto the first factor and $\Pi_{\Gamma}$ and $\pi_{\Gamma}$ are projections to the quotient under the respective left $\Gamma$-actions. Now, recall that the $\Gamma$-action on $\PSL\times G$ is given by
 $$\gamma\cdot (u,g)=(\gamma u, \rho(\gamma) g).$$
Since it commutes with the right action of $G$ on $\PSL\times G$ given by 
 $$(u,g)\cdot g'=(u,gg'),$$ 
we have a natural right action of $G$ on $\Y$. 
 
 \begin{lemma}\label{lemma:K_es_invariante_por_Gamma_y_por_G}
 Let $K$ be a compact subset of $T^1S$. Then $\mathcal{K}=\pi^{-1}(K)\subset \Y$ is invariant by the right action of $G$ on $\Y$.
 \end{lemma}

 \begin{proof} It suffices to prove that 
$$
\widetilde{\K}=\Pi_{\Gamma}^{-1}\big(\pi^{-1}(K)\big)\subset \PSL\times G.
$$
is invariant under both actions, the left action of $\Gamma$
and the right action of $G$. The $\Gamma$-invariance is straightforward, since $\widetilde{\K}$ is the preimage under $\Pi_{\Gamma}$ of $\K$. Notice that an alternative way to write $\mathrm{K}$ is 
 $$\widetilde{\K}=\pi_1^{-1}\big((\pi_{\Gamma})^{-1}(K)\big).$$
Now the $G$-invariance of $\widetilde{\K}$ is also straightforward, since it is the preimage under $\pi_1$ of a set in $\PSL$.
\end{proof}

 \begin{lemma}\label{lemma:cambio_de_horociclo_denso}
  Let $K$ be a compact subset of $T^1S$, and $\mathcal{K}=\pi^{-1}(K)\subset \Y$. Let $g\in G$. If $H_1$ is a subset of $\PSL\times \{g\}\subset \PSL\times G$ and $\gamma\in \Gamma$, let 
  $$H_1'=\{\, (\gamma u,g)\in \PSL\times G\, :\, (u,g)\in H_1 \, \}.$$ 
  If $\Pi_{\Gamma}(H_1)$ is $\varepsilon$-dense in $\mathcal{K}$, so is $\Pi_{\Gamma}(H'_1)$.
 \end{lemma}

 \begin{proof}
 If we define
 \begin{eqnarray*}
 H''_1&=&\{ \, \gamma\cdot (u,g)\in \PSL\times G\, :\, (u,g)\in H_1 \, \}\\
 &=&\{ \, (\gamma u,\rho(\gamma)g)\in \PSL\times G\, :\, (u,g)\in H_1  \, \},
 \end{eqnarray*}
 then clearly $\Pi_{\Gamma}(H_1)=\Pi_{\Gamma}(H''_1)$, which is $\varepsilon$-dense in $\K$. Consider $g'\in G$ such that $\rho(\gamma)gg'=g$. The right action of $G$ in $\Y$ preserves the natural distance $\mathbf{d}$, so $\Pi_{\Gamma}(H''_1)g'$ is $\epsilon$-dense in $\K g'$. By Lemma~\ref{lemma:K_es_invariante_por_Gamma_y_por_G}, $\K g'=\K$. On the other hand, $\Pi_{\Gamma}(H''_1)g'=\Pi_{\Gamma}(H'_1)$, so $\Pi_{\Gamma}(H'_1)$ is $\varepsilon$-dense in $\K$, as claimed.
\end{proof}

We will now prove Lemma~\ref{lemma:epsilon_densidad_en_K}.

\begin{proof}[Proof of Lemma~\ref{lemma:epsilon_densidad_en_K}] 
Let $\mathcal{K}$, $\varepsilon$ and $\M$ be as in the statement of the Lemma. 

Let $g\in G$ and $H_0 \subset \PSL\times \{g\}$ a horocyclic orbit that projects, via $\Pi_{\Gamma}$, onto a horocyclic orbit contained in $\M$. By Theorem~\ref{thm:principal}, there are many horocyclic orbits in $\PSL\times\{g\}$ whose projection under $\Pi_{\Gamma}$ is dense in $\Y$. So we can consider a compact segment $H_1$ of a horocyclic orbit in $\PSL\times \{g\}$ such that $\Pi_{\Gamma}(H_1)$ is $\frac{\varepsilon}{2}$-dense in $K$. Let $R$ be the length of $H_1$.

Since the horocycle flow $h_s$ in the hyperbolic plane $\mathbb{H}$ is continuous (and the isometry group $\PSL$ acts transitively in $T^1\mathbb{H}$), there exists $\delta>0$ such that 
$$u,v\in \PSL,\ d(u,v)<\delta\ \Longrightarrow \ \forall s\in [-R,R],\ d(h_su,h_sv)<\frac{\varepsilon}{2}.$$
 
Recall that since $\M$ is a minimal invariant set for the horocycle flow on $\Y$, it is a minimal set over a closed horocycle, as a consequence of Theorem~\ref{thm:principal}. We know that $\Pi_{\Gamma}(H_0)\subset \M$, so $\pi(\Pi_{\Gamma}(H_0))$ is a closed horocyclic orbit in $T^1S$. There exists $T\in \R$ such that for all $t<T$ the set $g_t(\pi(\Pi_{\Gamma}(H_0)))$ is $\delta$-dense in $K$. Of course
$$g_t(\pi(\Pi_{\Gamma}(H_0)))=\pi(g_t(\Pi_{\Gamma}(H_0)))=\pi(\Pi_{\Gamma}(g_t(H_0))),$$
if we allow for the abuse of notation which consists in calling $g_t$ the geodesic flows in $T^1S$, in $\Y$ and in $\PSL\times G$.

We claim that $g_t(\Pi_{\Gamma}(H_0))$ is $\varepsilon$-dense in $\mathcal{K}$ for $t< T$. This implies that $g_t(\M)$ is $\varepsilon$-dense in $\K$, as desired. Let us prove this claim.

We fix $t<T$. We know that $g_t(\pi(\Pi_{\Gamma}(H_0)))$ is $\delta$-dense in $K$, which means that any point in $K$ is at distance less than $\delta$ from $g_t(\pi(\Pi_{\Gamma}(H_0)))$. This is true of any particular point $u_1\in \pi(\Pi_{\Gamma}(H_1))\cap K$. Such a point exists because $K$ contains a ball of radius greater than $\varepsilon$ and $\pi(\Pi_{\Gamma}(H_1))$ is $\frac{\varepsilon}{2}$-dense in $K$. So for a fixed $u_1$ there exists $v_0\in g_t(\pi(\Pi_{\Gamma}(H_0)))$ such that $d(u_1,v_0)<\delta$.
Using the fact that 
$$(\pi\circ \Pi_{\Gamma})_{|(\PSL\times\{g\})}:\PSL\times \{g\} \to T^1 S$$ 
is a covering and a local isometry, there exist $v\in g_t(H_0)$ and $u$ which is some preimage of $u_1$ under $(\pi\circ \Pi_{\Gamma})_{|(\PSL\times\{g\})}$ such that $d(u,v)<\delta$ in $\PSL\times \{g\}$. Notice that $u$ belongs to the preimage of $\pi(\Pi_{\Gamma}(H_1))$ under $(\pi\circ \Pi_{\Gamma})_{|(\PSL\times\{g\})}$, namely, to a horocycle segment (of length $R$)
$$H_1'=\{\, (\gamma w,g)\in \PSL\times G\, :\, (w,g)\in H_1 \, \},$$ 
for some $\gamma\in \Gamma$.

We have chosen $\delta$ so that $d(h_su,h_sv)<\frac{\varepsilon}{2}$, for all $s\in [-R,R]$. As $s$ varies in $[-R,R]$, $h_su$ covers all the points of $H'_1$ (and more). On the other hand, $h_sv$ is always a point in $g_t(H_0)$. This means that, in $\PSL\times \{g\}$, $g_t(H_0)$ is $\frac{\varepsilon}{2}$-dense in $H'_1$.

The projection $\Pi_{\Gamma}$ is 1-Lipschitz from $\PSL\times G$ to $\Y$. By Lemma~\ref{lemma:cambio_de_horociclo_denso}, $\Pi_{\Gamma}(H_1')$ is $\frac{\varepsilon}{2}$-dense in $\mathcal{K}$. And for any point in this set, there is a point in $\Pi_{\Gamma}(g_t(H_0))$ that is at distance less than $\frac{\varepsilon}{2}$ from it. Therefore 
$$\Pi_{\Gamma}(g_t(H_0))=g_t(\Pi_{\Gamma}(H_0))$$ 
is $\varepsilon$-dense in $\mathcal{K}$, and so is the larger set $g_t(\M)$, as claimed.
\end{proof}

The last ingredient we need to prove Theorem~\ref{thm:convergence_in_Hausdorff_distance} is the following elementary fact about compact metric spaces.

\begin{lemma}\label{lemma:podemos_cambiar_la_distancia_sin_perder_epsilon_densidad}
Let $X$ be a compact metrizable space and $d,d'$ two distances which induce the topology in $X$. For every $\varepsilon>0$ there exists $\varepsilon'>0$ such that if $A\subset X$ is $\varepsilon'$-dense in $(X,d')$, then $A$ is $\varepsilon$-dense in $(X,d)$.
\end{lemma}

\begin{proof}
To see this, consider a finite covering $F=\{U_i\}_{i=1}^n$ by open balls of radius $\varepsilon/2$ in $(X,d)$. For every $x\in X$ there is an open ball $U'_x$, in $(X,d')$, centred at $x$, which is contained in the ball centred at $x$ of radius $\frac{\varepsilon}{4}$ in $(X,d)$. Since $\{U'_x\}_{ x\in X}$ is an open covering of $X$, it has a finite subcovering $F'=\{U'_j\}_{j=1}^m$.

Notice that for every $i \in \{1,\ldots,n\}$ there exists an $j\in\{1,\ldots,m\}$ such that $U'_j\subset U_i$.

Let $\varepsilon'$ be the minimum radius, in the distance $d'$, of the balls in $F'$. If $A\subset X$ is $\varepsilon'$-dense in $(X,d')$, it intersects all elements $U'_j$ of $F'$. Since any $U_i \in F$ contains an element of $F'$, $A\cap U_i\neq \emptyset$. Any ball of radius $\varepsilon$ in $(X,d)$, contains an element of $F$, so $A$ is $\varepsilon$-dense in $(X,d)$.
\end{proof}

Now we will prove Theorem~\ref{thm:convergence_in_Hausdorff_distance}.
\medskip

\begin{proof}
Let $\M\subset \Y$ be a minimal set for the horocycle flow. Let $d$ be any distance which induces the topology on the compact space $\widehat{\Y}$. Let 
$$O=\bigcup_{ \xi \in \widehat{\Y}-\Y} B( \xi,\varepsilon/2).$$
It is an open set, so its complement $O^c$ is a compact subset of $\Y$. Let $K=\pi(O^c)\subset T^1S$ and $\K=\pi^{-1}(K) \supset O^c$. Then Lemma~\ref{lemma:epsilon_densidad_en_K}, together with Lemma~\ref{lemma:podemos_cambiar_la_distancia_sin_perder_epsilon_densidad}, guarantee that there exists $T\in\R$ such that, for all $t<T$, $g_t(\M)$ is $\frac{\varepsilon}{2}$-dense in $\K$ and therefore in $O^c$. This readily implies that it is $\varepsilon$-dense in $\widehat{\Y}=O\cup O^c$.
\end{proof}

In \cite{ADMV2}, the minimality of the foliated horocycle flow in a Riemannian foliated compact manifold was derived from two measure theoretic theorems: Moore's ergodicity theorem (which is also used in the proof of Theorem~\ref{thm:principal}) and a theorem of unique ergodicity by Y. Coudène \cite{Coudene}. Here the dynamical behaviour of the minimal sets for the laminated horocycle flow in a McCord solenoidal surface of finite type is derived from Sarnak's theorem about the distribution of closed horocycles with large period. Both works deal with the interplay between topological and measurable dynamics of foliated or laminated 
spaces, an area of growing interest.

\bibliography{solenoides}{}
\bibliographystyle{amsplain}

\end{document}